\documentclass[a4paper]{article}
\usepackage{float}
\usepackage{amsmath,amsfonts,amssymb}
\usepackage{a4wide,amsthm}

\usepackage{graphicx}
\usepackage{caption,subcaption}
\usepackage{fonttable}

\DeclareFontFamily{OMX} {MnSymbolE}{}
\DeclareFontShape{OMX}{MnSymbolE}{m}{n}{
  <-6> MnSymbolE5
  <6-7> MnSymbolE6
  <7-8> MnSymbolE7
  <8-9> MnSymbolE8
  <9-10> MnSymbolE9
  <10-12> MnSymbolE10
  <12-> MnSymbolE12}{}
\DeclareFontShape{OMX}{MnSymbolE}{b}{n}{
  <-6> MnSymbolE-Bold5
  <6-7> MnSymbolE-Bold6
  <7-8> MnSymbolE-Bold7
  <8-9> MnSymbolE-Bold8
  <9-10> MnSymbolE-Bold9
  <10-12> MnSymbolE-Bold10
  <12-> MnSymbolE-Bold12}{}

\DeclareSymbolFont{MnSyE} {OMX} {MnSymbolE}{m}{n}

\DeclareMathDelimiter{\llangle}{\mathopen}{MnSyE}{116}{MnSyE}{120}
\DeclareMathDelimiter{\rrangle}{\mathopen}{MnSyE}{121}{MnSyE}{125}

\DeclareMathOperator\ima{im}   %used for image of a map

\DeclareMathOperator\dd{d}
\DeclareMathOperator\diag{\diag}

\DeclareMathOperator{\Res}{Res}
\DeclareMathOperator{\End}{End}
\DeclareMathOperator{\id}{id}

\makeatletter
\renewcommand\paragraph{\@startsection{paragraph}{4}{\z@}%
            {-2.5ex\@plus -1ex \@minus -.25ex}%
            {1.25ex \@plus .25ex}%
            {\normalfont\normalsize\bfseries}}
\makeatother
\begin{document}

\title{Darboux and Calapso transforms of meromorphically isothermic surfaces}
\author{Andreas Fuchs}
\date{Jan 17, 2019}

\maketitle

\newcommand{\Com}{\mathbb C}
\newcommand{\G}{\mathbb G}
\newcommand{\Hy}{\mathbb H}
\newcommand{\Nat}{\mathbb N}
\newcommand{\Q}{\mathbb Q}
\newcommand{\Pro}{\mathbb P}
\newcommand{\R}{\mathbb R}
\newcommand{\Z}{\mathbb Z}

\newcommand{\Rea}{\Re}
\newcommand{\Ima}{\Im}

\newcommand{\A}{\mathcal A}
\newcommand{\cB}{\mathcal B}
\newcommand{\cC}{\mathcal C}
\newcommand{\cE}{\mathcal E}
\newcommand{\Fr}{\mathcal F}
\newcommand{\I}{\mathcal I}
\newcommand{\J}{\mathcal J}
\newcommand{\K}{\mathcal K}
\newcommand{\Li}{\mathcal L}
\newcommand{\mon}{\mathcal M}
\newcommand{\N}{\mathcal N}
\newcommand{\Ro}{\mathcal R}
\newcommand{\Ord}{\mathcal O}
\newcommand{\ord}{\scalebox{0.6}{$\mathcal O$}}
\newcommand{\cS}{\mathcal S}
\newcommand{\T}{\mathcal T}
\newcommand{\cU}{\mathcal U}
\newcommand{\cV}{\mathcal V}
\newcommand{\Zc}{\mathcal Z}
\newcommand{\cZ}{\mathcal Z}

\newcommand{\afr}{\mathfrak {a}}
\newcommand{\bfr}{\mathfrak {b}}
\newcommand{\cfr}{\mathfrak {c}}
\newcommand{\Cfr}{\mathfrak {C}}
\newcommand{\ffr}{\mathfrak {f}}
\newcommand{\chfr}{\hat{\mathfrak {c}}}
\newcommand{\fhfr}{\hat{\mathfrak {f}}}
\newcommand{\gfr}{\mathfrak {g}}
\newcommand{\so}{\mathfrak {so}}
\newcommand{\gli}{\mathfrak {gl}}
\newcommand{\nfr}{\mathfrak {n}}
\newcommand{\ort}{\mathfrak {o}}
\newcommand{\pro}{\mathfrak {p}}
\newcommand{\TT}{\mathfrak T}
\newcommand{\Zfr}{\mathfrak Z}

\newcommand{\Ipl}{\left\llangle} % inner product left bigger
\newcommand{\Ipr}{\right\rrangle} % inner product right bigger

\newcommand{\ipl}{\Ipl} % inner product left
\newcommand{\ipr}{\Ipr} % inner product right

\newcommand{\Rmn}{\R^{n+2}_1}

\newcommand{\lqb}{\lim_{q\rightarrow \sip}}
\newcommand{\lqbc}{\lim_{\comp\ni q\rightarrow \sip}}

\newcommand{\Ta}{\afr}
\newcommand{\Tb}{\bfr}
\newcommand{\bs}{\backslash}
\newcommand{\nei}{\cU}
\newcommand{\sip}{s}
\newcommand{\bsb}{\backslash\{\sip\}}
\newcommand{\cub}{\cup\{\sip\}}
\newcommand{\wb}{\cup\{b\}}
\newcommand{\bd}{\ell}
\newcommand{\tang}{\mathcal T}

\newcommand{\gatr}[2]{#1\!\ltimes \! #2} % gauge transform by map 
\newcommand{\Gtr}[3]{#1 \ltimes_{\!#2} #3}
\newcommand{\proj}{\mathfrak p}
\newcommand{\pb}[1]{\boldsymbol{#1}}

\newcommand{\comp}{\mathcal C}

%---- O J E C T S  A S S O C I A T E D  T O  S U R F A C E 
\newcommand{\Qend}{\mathfrak Q}   %----- ENDOMORPHISM RELATED TO POLARIZATION
\newcommand{\Qendc}{\mathfrak Q}   %----- ENDOMORPHISM RELATED TO POLARIZATION FOR CURVES
\newcommand{\Qho}{\mathcal H}   %----- HOPF DIFFERENTIAL
\newcommand{\Qua}{Q}   %----- HOLOMORPHIC QUADRATIC DIFFERENTIAL
\newcommand{\Pol}{Q}   %----- POLARIZATION
\newcommand{\of}{\go} % 1-form associated to surface or curve
\newcommand{\Of}{\gO} % o-valued lift of it
\newcommand{\curdir}{\eta}
\newcommand{\tandir}{\gs}
\newcommand{\tandira}{\mu}
\newcommand{\tandirb}{\nu}
\newcommand{\curdirc}{\eta}
%---- 1 F O R M S
\newcommand{\pf}{\psi} % symbol for pole form
\newcommand{\Pf}{\Psi} % symbol for lift of pole form
\newcommand{\ppf}{\xi} % pure pole form
\newcommand{\ppfr}{\ppf^\Rea}
\newcommand{\ppfi}{\ppf^\Ima}
\newcommand{\Ppfr}{\Ppf^\Rea}
\newcommand{\Ppfi}{\Ppf^\Ima}
\newcommand{\Ppf}{\Xi} % orthogonal lift of pure pole form
\newcommand{\aof}{\psi} % arbitrary one form
\newcommand{\Aof}{\Psi} % lift of arbitrary one form
%---- 1 F O R M S
\newcommand{\sing}{\mathcal S}

%---- C O N S T A N T    B A S I S    O F    R^(n+2)
\newcommand{\baszer}{o}
\newcommand{\basinf}{\boldsymbol{\iota}}
\newcommand{\bastan}{\mathfrak t}
\newcommand{\bastanu}{\mathfrak t_u}
\newcommand{\bastanv}{\mathfrak t_v}
\newcommand{\basnor}{\mathfrak n}
%-------------------------------------------------

\newcommand{\la}{\langle}
\newcommand{\ra}{\rangle}
\newcommand{\La}{\left\langle}
\newcommand{\Ra}{\right\rangle}
\newcommand{\LiSe}{\pounds}

\newcommand{\pr}[1]{{#1}^{\prime}}
\newcommand{\dpr}[1]{{#1}^{\prime\prime}}
\newcommand{\ti}[1]{\tilde{#1}}
\newcommand{\dti}[1]{\tilde{\tilde{#1}}}
\newcommand{\p}{\partial}
\newcommand{\fh}{{\hat f}}
\newcommand{\Fh}{{\hat F}}

\newcommand{\sh}{{\hat s}}
\newcommand{\Sh}{{\hat S}}

\newcommand{\hyfr}{{g}}

\newtheorem{prop}{Proposition}[section]
\newtheorem{theorem}[prop]{Theorem}
\newtheorem{lemma}[prop]{Lemma}
\newtheorem{corollary}[prop]{Corollary}
\newtheorem{defi}[prop]{Definition}
\newtheorem{defiandlemma}[prop]{Definition and Lemma}
\newtheorem{remark}[prop]{Remark}
\newtheorem{example}{Example}

\newcommand{\ga}{\alpha}
\newcommand{\gb}{\beta}
\newcommand{\gga}{\gamma}
\newcommand{\gd}{\delta}
\newcommand{\gD}{\Delta}
\newcommand{\gth}{\theta}
\newcommand{\vth}{\vartheta}
\newcommand{\vphi}{\varphi}
\newcommand{\gk}{\kappa}
\newcommand{\gl}{\lambda}
\newcommand{\gL}{\Lambda}
\newcommand{\go}{\omega}
\newcommand{\gO}{\Omega}
\newcommand{\gomi}{\omicron}
\newcommand{\gs}{\sigma}
\newcommand{\gS}{\Sigma}
\newcommand{\gep}{\epsilon}
\newcommand{\gG}{\Gamma}

\textbf{MSC 2010}: 35B40, 51B10, 58K10, 58J72

\begin{abstract}
We consider those simply connected isothermic surfaces for which their Hopf differential factorizes into a real function and a meromorphic quadratic differential that has a zero or pole at some point, but is nowhere zero and holomorphic otherwise. Upon restriction to a simply connected patch that does not contain the zero or pole, the Darboux and Calapso transformations yield new isothermic surfaces. We determine the limiting behaviour of these transformed patches as the zero or pole of the meromorphic quadratic differential is approached and investigate whether they are continuous around that point.
\end{abstract}

\section{Introduction}
In differential geometry transformations of surfaces have contributed considerably in revealing the structures of various surface classes and relations between them. Already in the nineteenth century geometers discovered methods to construct new surfaces from a given one while preserving some geometrical properties, such as the Combescure transformation, which preserves tangent planes up to translations, and the Ribaucour transformation, which preserves curvature lines and an enveloped sphere congruence, see \cite{bia22} and \cite{eis62}, or \cite{daj03} for a more modern treatment. While these transformations exist for arbitrary smooth surfaces in Euclidean space, others are defined only on subclasses of smooth surfaces, such as the B\"acklund and Lie transformations of pseudospherical surfaces (see \cite{eis60}) in Euclidean geometry or the Christoffel, Darboux and Calapso transformations of isothermic surfaces in M\"obius geometry (see \cite{bur06,jer03}) and Lie applicable surfaces in Lie sphere geometry (see \cite{pem16}). During the last three decades, many of these surfaces have been shown to constitute integrable systems by relating their transformations to a pencil of flat connections, as summarized in \cite{bur17}, cf. \cite{ter00,bur10a,bur06,bur93,fer96}. Thereby, their rich transformation theory is placed into common ground and the highly developed tools of integrable systems theory are made available. A discrete version of integrability serves as the guideline for how to define discrete analogues of certain smooth surface classes \cite{bob07}.

Transformations of surfaces can also be used to shed a different light on representation formulas. For example, the Weierstrass representation of a minimal surface can be interpreted as the Goursat transform of its Weierstrass data, as explained in \cite{jer03,hon17}, and Bryant's representation \cite{bry87} of CMC-1 surfaces in hyperbolic space can be related to the Darboux transformation of isothermic surfaces, see \cite{jer01}. Recently, these and further representation formulas have been linked to the transformation theory of $\gO$-surfaces in \cite{pem16}.

We are interested in the class of smooth isothermic surfaces, classically characterized by the local existence of conformal curvature line coordinates, away from umbilics. For these surfaces, the transformation theory is defined only locally and away from problematic umbilics. In other words, it is well defined only on the subclass that contains those simply connected isothermic surfaces which can be covered by conformal curvature line coordinate charts. These surfaces, which we call simple isothermic, can equivalently be characterized by the existence of a factorization of the Hopf differential of the surface into a real factor and a nowhere zero holomorphic quadratic differential (cf. \cite{bur06,jer03,smy04}). In particular, a simple isothermic surface cannot contain an umbilic of nonzero index and hence cannot have the topology of the sphere by the Poincar\'e-Hopf theorem (cf. \cite{hop89}).

Arguably, this limitation of the transformation theory is not quite satisfactory though as many classical examples of isothermic surfaces are not simple isothermic. For instance, a triaxial ellipsoid, surfaces of revolution which smoothly intersect the axis of revolution, and minimal surfaces with appropriate Weierstrass data all contain umbilics with nonzero index. However, all these examples fall into the class of meromorphically isothermic surfaces, that is, their Hopf differential factorizes into a real factor and a globally defined meromorphic quadratic differential. The globally existing factorization makes this class of surfaces a promising candidate for an extension of the simple isothermic transformation theory.

For a given meromorphically isothermic surface which is not simple isothermic, the transformation theory is still not defined globally. But after removing the zeros and poles of the meromorphic quadratic differential from the domain of the surface and then passing to the universal cover, one obtains a simple isothermic surface as a branched covering of the original surface. One may then ask \textit{which transforms of this simple isothermic covering surface descend to surfaces that qualify as transforms of the original meromorphically isothermic surface}. More precisely, for a meromorphically isothermic surface $\la f\ra$ with domain $M$, denote by $\sing\subset M$ the set of zeros and poles of the meromorphic quadratic differential $\Qua$ obtained from the factorization of the Hopf differential. Then the pullback $\la \pb f\ra$ of that surface to the universal cover $\pb M\bs\sing$ of $M\bs\sing$ is simple isothermic. The fundamental group of $M\bs\sing$ acts on the space of transforms of $\la \pb f\ra$ and one question is whether there are transforms which are invariant under this action. These would then qualify as transforms of the original surface $\la f\ra$ restricted to $M\bs\sing$. A second question is how the transforms of $\la \pb f\ra$ behave as one approaches the zeros and poles $\sing$ of $\Qua$. In particular, do the transforms have limits at $\sing$?

In this work, we address these questions for the Darboux and Calapso transforms of a simply connected meromorphically isothermic surface where the meromorphic quadratic differential $\Qua$ has a pole of first or second order at an interior point $\sip $, but is nowhere zero and holomorphic otherwise. According to the local Carath\'eodory conjecture \cite{gui08}, these are the only poles that can occur on an isothermic surface. We investigate for which parameters the Darboux and Calapso transforms of the pullback of that surface to the universal cover of $\gD\bsb$ are invariant under the action of the fundamental group and what their limiting behaviour at $\sip$ is. We also briefly treat the case where $\Qua$ has a zero and argue that it is conceptually much simpler than a pole of $\Qua$. Our results about the limits of transforms at a pole of $\Qua$ are generalizations of the corresponding results for transforms of polarized curves, found in \cite{fu18p}, to surfaces. The action of the fundamental group on the space of Darboux and Calapso transforms however is a novel feature, which is specific to surfaces.

In Sect \ref{section:stateofart}, we present a hierarchy of definitions of isothermic surfaces in the conformal $n$-dimensional sphere $S^n$, including meromorphically isothermic and simple isothermic surfaces. We then introduce our formulation of the Darboux and Calapso transformations of simple isothermic surfaces using the concept of primitives of the pencil  $(\dd+\gl\of)_{\gl\in\R}$ of flat connections associated to a simple isothermic surface. These primitives $\gG_p(\gl\of)$ are M\"ob$(S^n)$-valued maps characterized by
$$\gG_p(\gl\of)(\dd+\gl\of)(\gG_p(\gl\of))^{-1}=\dd,~~~\gG_p^{~p}(\gl\of)=\id,$$
for any point $p$ in the domain. A $\gl$-Darboux transform of $\la f\ra$ is obtained by acting with $(\gG_p(\gl\of))^{-1}$ on a point $\la \fh_p\ra\in S^n$ and a $\gl$-Calapso transform is the product of $\gG_p(\gl\of)$ with $\la f\ra$. 

These constructions of the Darboux and Calapso transformations cannot be extended to meromorphically isothermic surfaces. The family of flat connections $\dd+\gl\of$ exists only away from the poles of the meromorphic quadratic differential $\Qua$ and, in general, they have nonzero monodromy around poles. Therefore, their primitives do not exist globally, not even away from the poles of $\Qua$. But the primitives along closed loops generate the monodromy group. In Prop \ref{prop:monodromy_dar_cal_trafo}, we use this to derive actions of the fundamental group of a meromorphically isothermic surface with removed poles on the spaces of Darboux and Calapso transforms of its simple isothermic covering. We conclude this section with some results about the case of a meromorphically quadratic differential with a zero.

In Sect \ref{section:limits_of_prims_surf}, we develop tools to analyse the limiting behaviour of the primitives and the monodromy of the flat connections associated to a meromorphically isothermic surface. In particular, we introduce a class of singular flat connections which have a pole at some point in the domain. We then analyse the limiting behaviour of primitives of such connections at the pole and investigate the structure of their monodromy. 

In Sect \ref{section:surf_pole_fo}, we consider the case of a simply connected meromorphically isothermic $\la f\ra$ with a meromorphic quadratic differential that has pole of first order at a point $\sip$, but is nowhere zero and holomorphic otherwise. Using the results of Sect \ref{section:limits_of_prims_surf}, we show that all Darboux and Calapso transforms of the simple isothermic covering of $\la f\ra$ have a limit at $\sip$. Moreover, we show that no Calapso transform is continuous around the pole and for every value of the spectral parameter $\gl$ there is precisely one Darboux transform that is continuous around and at the pole

The case of a pole of second order is treated in Sect \ref{section:surf_pole_so}. We can again apply the results of Sect \ref{section:limits_of_prims_surf} to show that in this case a generic Darboux transform still has a limit at the pole, but there are Darboux transforms which do not converge as one approaches the pole. The Calapso transforms either have a limit point or a limit sphere, depending on the value of the spectral parameter $\gl$. Moreover, for positive $\gl$ we show that no Calapso transform and almost no Darboux transform is continuous around the pole.

%Moreover, for positive $\gl$ we compute the action of the fundamental group. It turns out that no $\gl$-Calapso transform is invariant and also a generic Darboux transform is not invariant. 
%However, when $j\sqrt{1-2\gl}$ is an integer for some $j\in\mathbb N$, the pushforward of any $\gl$-Calapso transform and any $\gl$-Darboux transform to the $j$-fold cover of $\gD\bsb$ is well-defined.

\textit{Acknowledgements:} The author would like to thank C. Bohle for his comments on the monodromy of the family of flat connections. Special thanks go to U. Hertrich-Jeromin for his numerous discussions, valuable feedback and his support as supervisor of the author's doctoral thesis \cite{fu18}, on which this work is based. It has been supported by the Austrian Science Fund (FWF) and the Japan Society for the Promotion of
Science (JSPS) through the FWF/JSPS Joint Project grant I1671-N26 ``Transformations and Singularities''.

\section{Darboux and Calapso transforms of isothermic surfaces}\label{section:stateofart}
The characterization of isothermic surfaces is invariant under conformal transformations. We thus study them in M\"obius geometry, the geometry of the group of conformal transformations of the $n$-dimensional sphere $S^n$ equipped with its standard conformal structure\footnote{More generally, one may consider isothermic submanifolds of symmetric $R$-spaces, as explained in \cite{bur11}}. We use the projective model of M\"obius geometry to identify $S^n$ with the projectivization $\Pro(\Li^{n+1})$ of the light cone $\Li^{n+1}$ in $(n+2)$-dimensional Minkowski space $\Rmn$ (cf. \cite[Ch 1]{jer03} or \cite[Ch 1]{bur06}). We always assume a surface $\la f\ra$ in $S^n$ to be orientable and immersed, such that it can be realized as a conformal immersion of a Riemann surface $M$ into $\Pro(\Li^{n+1})$. The symbol $\la f\ra$ indicates that it can equivalently be described by pointwise taking the linear span $\la\cdot\ra$ of an immersion $f:M\rightarrow\Li^{n+1}\subset\Rmn$, called a lift of $\la f\ra$. When $N$ is a spacelike unit normal field of any lift $f$ of $\la f\ra$, then, for any function $a$, we say that $N+a f$ is an equivalent normal field. The equivalence class $N+\la f\ra$ is then a normal field of $\la f\ra$ and we call any of its representatives a \textit{lift} of $N+\la f\ra$. The normal fields of $\la f\ra$ are sections of the normal bundle $\mathcal N_{\la f\ra}$ of $\la f\ra$.

The pullback of the Minkowski inner product of $\Rmn$ via a lift $f$ induces a Riemannian metric on $M$. Different lifts of $\la f\ra$ induce conformally equivalent metrics. In this way, a surface in $S^n$ equips its domain manifold $M$ with a conformal structure, which we always assume to agree with that induced by the complex structure of $M$. 

Classically, a surface in $S^n$ is isothermic if conformal curvature line coordinates exist locally around every non-umbilic point. However, the transformation theory of isothermic surfaces is defined only on the subclass of simple isothermic surfaces. To define the latter, we need the notion of the \textit{Hopf differential} $\Qho$ of a surface in $S^n$ given by
\begin{align}
\Qho^f_{N+\la f\ra}(v,w):=-\ipl \p f(v),\p N(w)\ipr \label{eq:hopf_diff}
\end{align}
for all lifts $f$ of $\la f\ra$, all lifts $N$ of all normal fields $N+\la f\ra$ and all complex sections $v,w$ of $ TM_\Com$. Here, we have extended the Minkowski inner product $\ipl\cdot,\cdot\ipr$ bilinearly to $\Com^{n+2}$ and denote by $\p$ the Dolbeault operator of the Riemann surface $M$, such that $\|\p f\|^2=0$. 

While the classical definition of isothermicity is purely local, the following definitions link the geometry of a surface to a globally existing meromorphic quadratic differential on $M$.
\begin{defi}\label{defi:iso_surf}
A {\upshape polarized surface in }$S^n$ is a pair $(\la f\ra,\Qua)$ of a surface $\la f\ra:M\rightarrow S^n$ and a meromorphic quadratic differential $\Qua\in (TM^*_\Com)^2$ on $M$. A polarized surface $(\la f\ra,\Qua)$ in $S^n$ is
\begin{enumerate}
	\item {\upshape meromorphically isothermic} if $M$ is connected and the Hopf differential $\Qho$ of $\la f\ra$ factorizes into $\Qua$ and a (real) bilinear map $\gk$, that is,
	\begin{equation}
	\Qho^f_{N+\la f\ra}(v,w)=\gk(f,N+\la f\ra) ~\Qua(v,w);
	\label{eq:factorization_Hopf_diff}
	\end{equation} 
	\item {\upshape holomorphically isothermic} if it is meromorphically isothermic and additionally $\Qua$ is holomorphic;
	\item {\upshape simple isothermic} if it is holomorphically isothermic, $M$ is simply connected and $\Qua$ is nowhere zero.
\end{enumerate}
\end{defi}
As proved in \cite{fu18}, we have the proper inclusions
\begin{equation}\label{eq:inclusions_isothermic}
\begin{split}
&\{simple~isothermic~surfaces\}\subset \{holomorphically~isothermic~ surfaces\}\\
\subset \{& meromorphically~isothermic~ surfaces\}\subset\{classically~isothermic~ surfaces\}.
\end{split}
\end{equation}

For a meromorphically isothermic surface $(\la f\ra,\Qua)$ with domain $M$, if the restriction of $\Qua$ to a curve $c$ in $M$ is real, then also the restriction of the Hopf differential $\Qho$ to $c$ is real, such that $c$ is a curvature line of $\la f\ra$ (cf. \cite[Ch VI, Sect 1.2]{hop89}). Therefore, if $\la f\ra$ is not totally umbilic, the differential $\Qua$ for which $(\la f\ra,\Qua)$ is meromorphically isothermic is determined up to nonzero, constant rescalings. If on the other hand $\la f\ra$ is totally umbilic, then $(\la f\ra,\Qua)$ is meromorphically isothermic for any meromorphic quadratic differential $\Qua$ on $M$.

The central object for the transformation theory of a simple isothermic surface is its associated 1-form with values in the Lie algebra of the M\"obius group. In the projective model of M\"obius geometry, the action of the M\"obius group on $S^n$ is identified with the action of the projective Lorentz group $\Pro O(\Rmn)$ on $\Pro(\Li^{n+1})$. We view $\Pro O(\Rmn)$ as a subgroup of $\Pro GL(\Rmn)$, which in turn is a submanifold of $\Pro \End(\Rmn)$. Taking the linear span $\la \cdot\ra$ of elements of $\End(\Rmn)$ then provides a diffeomorphism from the group $O^+(\Rmn)\subset GL(\Rmn)\subset \End(\Rmn)$ of orthochronous Lorentz transformations to $\Pro O(\Rmn)$. Similarly, we view the Lie algebra $\pro\ort(\Rmn)$ as a subspace of the tangent space $\End(\Rmn)/\R\,\id$ of $\Pro \End(\Rmn)$ at the identity. The differential of $\la\cdot\ra$ at the identity then restricts to an isomorphism of Lie algebras
\begin{equation}\label{eq:iso_liealg}
\begin{split}
\dd_{\id} \la\cdot\ra:~\End(\Rmn)\supset \ort(\Rmn)&\rightarrow \pro\ort(\Rmn)\subset \End(\Rmn)/\R\,\id,\\
v\wedge w&\mapsto v\wedge w+\R\,\id,
\end{split}
\end{equation}
where we further identify $\gL^2(\Rmn)$ with $\ort(\Rmn)$ via
$$(v\wedge w)(x)=\ipl v,x\ipr w-\ipl w,x\ipr v$$
for $v,w,x\in\Rmn$. We denote the image of an element $\Aof\in\ort(\Rmn)$ under $\dd_{\id}\la \cdot\ra$ in $\pro\ort(\Rmn)$ by $\Aof+\R\,\id$ and call $\Aof$ the \textit{orthogonal lift} of $\Aof+\R\,\id$.

\begin{defi}\label{defi:eta}
	Let $f$ be a lift of a polarized surface $(\la f\ra,\Qua)$ with domain $M$. Away from the poles of $\Qua$, we define the $\pro\ort(\Rmn)$-valued {\upshape 1-form $\of$ associated to} $(\la f\ra,\Qua)$ by
	\begin{equation}
	\of=f\wedge \dd f\circ \Qend+\R\,\id.
	\label{eq:defieta}
	\end{equation}
	Here, $\Qend$ is the unique symmetric, trace-free endomorphism of $T M$ that satisfies
	\begin{equation}
	\frac 12 \Qua=\ipl \dd f\circ \Qend,\dd f\ipr^{(2,0)}\in TM^{(2,0)}\subset TM_\Com.
	\label{eq:rel_Qend_Q}
	\end{equation}
	By $\Of$ we denote the orthogonal lift of $\of$.
\end{defi}
Clearly, $\of$ is independent of the chosen lift and takes values in $\la f\ra\wedge \la f\ra^\perp+\R\,\id$, but depends on the choice of holomorphic quadratic differential $\Qua$. From the requirement that $\Qend$ be symmetric and trace-free, it follows that $\Qend$, linearly extended to $TM_\Com$, maps sections of $TM^{(1,0)}$ to sections of $TM^{(0,1)}$ and vice versa. Using \eqref{eq:rel_Qend_Q} we then find that it satisfies
\begin{equation}
 \bar\p f\circ\Qend=\bar \p f\, \frac{\Qua}{2\ipl \bar\p f,\p f\ipr}.
\label{eq:component_endomorphism}
\end{equation}
Here, the right hand side is the $\Rmn$-valued 1-form locally given by $f_{\bar z}\frac{\Qua_z}{2\ipl f_{\bar z},f_z\ipr}\dd z$ with $\Qua=\Qua_z\dd z^2$ in terms of an arbitrary holomorphic coordinate $z$.

The transformation theory of simple isothermic surfaces relies on the closedness of their associated 1-forms $\of$.
\begin{lemma}\label{lemma:closedness_of_oneform}
	Let $(\la f\ra,\Qua)$ be a polarized surface in $S^n$ with connected domain $M$ and holomorphic $\Qua$. Then $(\la f\ra,\Qua)$ is holomorphically isothermic if and only if its associated 1-form $\of$ is closed, which in turn is equivalent to the flatness of $\dd+\gl\of$ for all $\gl\in\R$.
\end{lemma}
\begin{proof}
Choose a local holomorphic coordinate $z$ and write $\of^{(1,0)}=\of_z\dd z$. Then
$$\dd \of=0~~\Leftrightarrow~~\bar\p\of^{(1,0)}+\p\of^{(0,1)}=0~~\Leftrightarrow~~\Rea(\bar\p\of^{(1,0)})=0~~\Leftrightarrow~~\Ima\left(\frac{\p \of_z}{\p \bar z}\right)=0.$$
A small computation using \eqref{eq:component_endomorphism} and holomorphicity of $\Qua=\Qua_z\dd z^2$ shows that
			$$\Ima\left(\frac{\p \of_z}{\p \bar z}\right)=f\wedge \sum_{i=1}^{n-2}N_i\frac{\Ima \left(\ipl f_{\bar z\bar z},N_i\ipr  \Qua_z\right)}{2\ipl f_z,f_{\bar z}\ipr },$$
	where $N_1,...,N_{n-2}$ are lifts of orthonormal normal fields. Using the definition \eqref{eq:hopf_diff} of the Hopf differential $\Qho$ of $\la f\ra$, we conclude that $\of$ is closed if and only if $\Qho$ is a real multiple of $\Qua$ which by definition is equivalent to $(\la f\ra,\Qua)$ being holomorphically isothermic.
	
	Since $[f\wedge \dd f\circ\Qend,f\wedge \dd f \circ\Qend]=0$, closedness of $\of$ is equivalent to flatness of $\dd+\gl\of$ for all $\gl\in\R$.
\end{proof}

Let $G$ be a Lie group with Lie algebra $\gfr$ and $\aof$ a $\gfr$-valued 1-form on a simply connected manifold $\gD$ such that $\dd+\aof$ is flat. For any $p\in \gD$, we denote by 
$$\gG_{p}(\aof):~\gD\rightarrow G,~~q\mapsto \gG_{p}^{~q}(\aof),$$
the unique primitive that satisfies (cf. \cite[Ch 3, Thm 6.1]{sha97})  
\begin{equation}
\dd\gG_{p}(\pf)=\gG_{p}(\pf)\pf,~~~~\gG_{p}^{~p}(\pf)=\id.
\label{eq:propsgG1}
\end{equation}
From the defining properties \eqref{eq:propsgG1} of $\gG_{p}(\aof)$ and its uniqueness it follows readily that
\begin{equation}
\forall p,q\in \gD:~~~~~\gG_{p}^{~q}(\aof)\gG_{q}(\aof)=\gG_{p}(\aof).
\label{eq:propsgG2}
\end{equation}
Therefore, the map $\gG^{~p}(\aof):\gD\ni q\mapsto \gG_q^{~p}(\aof)\in G$ is the composition of $\gG_{p}(\aof)$ with taking the inverse in $G$ and thus satisfies
$$\dd \gG^{~p}(\aof)=-\pf \gG^{~p}(\aof),~~~~\gG_p^{~p}(\aof)=\id.$$
Under a gauge transformation 
$$\aof\mapsto \gatr g \aof:=g^{-1}\aof g+g^{-1}\dd g$$
using a smooth map $g:\gD\rightarrow G$, the primitives $\gG_{p}(\aof)$ transform as
\begin{equation}
\gG_{p}(\aof)=g(p)\,\gG_{p}(\gatr g \aof)\,g^{-1}.
\label{eq:gauge_trafo_beha}
\end{equation}
We also remark that for the orthogonal lift $\Aof$ of a $\pro\ort(\Rmn)$-valued 1-form $\aof$, we have
\begin{equation}
\dd \la \gG_p(\Aof)\ra=\la \gG_p(\Aof)\ra \dd_{\id} \la\cdot\ra(\Aof)=\la \gG_p(\Aof)\ra \aof~~~\Rightarrow~~~\la \gG_p(\Aof)\ra=\gG_p(\aof).
\label{eq:primsofortholift}
\end{equation}

For the 1-form $\of$ associated to a simple isothermic surface, $\dd+\gl\of$ is flat for all $\gl\in\R$. We may thus define
\begin{defi}\label{defi:trafos_surf}
	Let $(\la f\ra,\Qua)$ be simple isothermic with domain $\gD$ and associated 1-form $\of$.
	
	For spectral parameter $\gl\in\R$ and $p\in \gD$, the pair $(\la f_{\gl,p}\ra,\Qua)$ with
	$$\la f_{\gl,p}\ra:=\gG_{p}(\gl\of)\la f\ra$$
	is called the $\gl$-{\upshape Calapso transform of $(\la f\ra,\Qua)$ normalized at} $p$.
	
	For $\gl\in\R^\times$, a point $p\in\gD$ and $\la \fh_p\ra\in S^n$ not lying on the image of $\la f_{\gl,p}\ra$, the pair $(\la \fh\ra,\Qua)$ with
	$$\la \fh\ra:=\gG^{~p}(\gl\of)\la \fh_p\ra$$
	is called the $\gl$-{\upshape Darboux transform} of $(\la f\ra,\Qua)$ with initial point $\la \fh(p)\ra=\la \fh_p\ra$.
\end{defi}
Using \eqref{eq:primsofortholift} and \eqref{eq:propsgG1}, one readily finds that our definitions agree with \cite[Def 1.12]{bur10}, up to M\"obius transformation,  and \cite[Def 1.7]{bur10} (cf. \cite[\S 8.6.13]{jer03} and \cite[\S 8.7.1]{jer03}). In particular, our condition that $\la \fh(p)\ra$ not lie on the image of $\la f_{\gl,p}\ra$ is equivalent to condition (1) of \cite[Def 1.7]{bur10}.

Replacing $\Qua$ by $\ti\gl\Qua$ with $\ti\gl\in\R^\times$ has the same effect as replacing $\gl$ by $\gl\ti\gl$. Thus, the sets of all Darboux and Calapso transforms of $(\la f\ra,\Qua)$ and those of $(\la f\ra,\ti\gl\Qua)$ are the same. We will use this in Sect \ref{section:surf_pole_so} and scale $\Qua$ conveniently.

The Darboux and Calapso transformations are well defined on all simple isothermic surfaces. In particular, any Darboux or Calapso transform of a simple isothermic surface is again simple isothermic\footnote{Although of geometric importance, we do not prove this fact here. The proof can be found in \cite{fu18}, see also \cite[\S 8.6.17, \S 8.7.3]{jer03}, \cite[Sect 1.3]{bur10}.}. There are several obstacles that one faces when trying to extend Def \ref{defi:trafos_surf} to meromorphically isothermic surfaces. First, at a pole of the meromorphic quadratic differential $\Qua$, the 1-form $\of$ associated to $(\la f\ra,\Qua)$ is not defined, but also has a pole there. Second, although at a zero of $\Qua$ the primitives of all $\gl\of$ are well defined and therefore also the maps $\gG^{~p}(\gl\of)\la\fh_p\ra$ of Def \ref{defi:trafos_surf}, these maps fail to immerse at that point and hence are not meromorphically isothermic again. A third obstacle arises when the domain $M$ is not simply connected. In that case, the 1-forms $\gl\of$ may have non-trivial monodromy and the primitives of $\gl\of$ are not globally defined on $M$. 

However, for any meromorphically isothermic surface $(\la f\ra,\Qua)$ with domain $M$, we can first remove the set $\sing$ of zeros and poles of $\Qua$ from $M$ and then pull back the restriction of $(\la f\ra,\Qua)$ to $M\bs\sing$ to the universal cover of $M\bs\sing$. This pullback $(\la \pb f\ra,\pb\Qua)$ is then a simple isothermic surface, such that all its Darboux and Calapso transforms in the sense of Def \ref{defi:trafos_surf} exist. One can then investigate whether these transforms can be pushed forward to $M\bs\sing$ and what their limiting behaviour at the points of $\sing$ is.

In this work, we partly answer these questions for the specific cases of a meromorphically isothermic surface $(\la f\ra,\Qua)$ on a simply connected domain $\gD$ where $\Qua$ has a pole of first or second order at an interior point $\sip $ of $\gD$, but is holomorphic and nowhere zero otherwise. 

We denote the universal cover of $\gD\bsb$ by\footnote{At this point, $\pb \gD\bsb$ should be understood as one symbol. Later, we will define the branched universal cover $\pb \gD:=(\pb \gD\bsb)\cub$, such that the symbol $\pb \gD\bsb$ for the universal cover may also be understood as $\pb \gD$ without $\sip$.} $\pb\gD\bsb$ and the covering map by $\proj:\pb\gD\bsb\rightarrow \gD\bsb$. Since $\proj$ is a local diffeomorphism, we can pullback any tensor $T$ defined on $\gD\bsb$ to $\pb\gD\bsb$ via $\proj$. We always use the corresponding bold symbol $\pb T$ for this pullback.

Whether or not the pushforward of a $\gl$-Darboux or $\gl$-Calapso transform of $(\la\pb f\ra,\pb\Qua)$ to the $j$-fold cover of $\gD\bsb$ exists depends on the monodromy group of $\gl\of$. In our case, the fundamental group of $\gD\bsb$ is isomorphic $\mathbb Z$ and so the monodromy group of $\gl\of$ with base point $P\in\gD\bsb$ is generated by the single element
\begin{equation}
\mon_P(\gl\of)=\gG_{0}^{~2\pi}(\gga^*\gl\of),
\end{equation}
where $\gga:[0,2\pi]\rightarrow \gD\bsb$ is a loop with base point $P$ and winding number $-1$ around $\sip$. We call $\mon_P(\gl\of)$ the \textit{monodromy of $\gl\of$ with base point} $P$.
\begin{prop}\label{prop:monodromy_dar_cal_trafo}
Let $(\la f\ra,\Qua)$ be a holomorphically isothermic surface with domain $\gD\bsb$, where $\gD$ is simply connected and $s$ is an interior point of $\gD$. Suppose further that $\Qua$ is nowhere zero on $\gD\bsb$. For $\gl\in\R^\times$ and $p\in\pb\gD\bsb$, let $(\la f_{\gl,p}\ra,\pb\Qua)$ be the $\gl$-Calapso transform of $(\la \pb f\ra,\pb\Qua)$ normalized at $p$ and let $(\la \fh\ra,\pb\Qua)$ be a $\gl$-Darboux transform of $(\la \pb f\ra,\pb\Qua)$. For $j\in\mathbb N$, 
\begin{enumerate}
	\item the pushforward of $\la \fh\ra$ to the $j$-fold cover of $\gD\bsb$ exists if and only if $\la \fh(p)\ra$ is invariant under $\big(\mon_{\proj(p)}(\gl\of)\big)^j$, and
	\item the pushforward of $\la f_{\gl,p}\ra$ to the $j$-fold cover of $\gD\bsb$ exists if and only if every point in the image of $\la f_{\gl,p}\ra$ is invariant under $\big(\mon_{\proj(p)}(\gl\of)\big)^j$. 
\end{enumerate}

\end{prop}
\begin{proof}
The cases $j\neq 1$ follow from the case $j=1$ if we replace $(\la f\ra,\Qua)$ by its pullback to the $j$-fold cover of $\gD\bsb$. Therefore, it is sufficient to prove the proposition for $j=1$. 

\begin{enumerate}
\item The pushforward of $\la \fh\ra$ to $\gD\bsb$ exists if and only if
\begin{equation}
\forall q,\ti q\in\pb \gD\bsb~~\text{with}~~\proj(q)=\proj(\ti q):~~~~~\la \fh(q)\ra=\la \fh(\ti q)\ra.
\label{eq:monodromy_cond_1}
\end{equation}
From Def \ref{defi:trafos_surf} and \eqref{eq:propsgG2}, we know that
$$\forall q,\ti q\in\pb \gD\bsb:~~~~~\la \fh(q)\ra=\gG_q^{~\ti q}(\gl\proj^*\of)\la \fh(\ti q)\ra.$$
If $q,\ti q$ are such that $\proj(q)=\proj(\ti q)$, then $\gG_q^{~\ti q}(\gl\proj^*\of)$ is an element of the monodromy group of $\gl\of$ with base point $\proj(q)$. Thus, \eqref{eq:monodromy_cond_1} holds if and only if $\la \fh(q)\ra$ is invariant under the monodromy group of $\gl\of$ with base point $\proj(q)$ for all $q\in\pb\gD\bsb$. But that is equivalent to the invariance of $\la \fh(p)\ra$ under the monodromy $\mon_{\proj(p)}(\gl\of)$ for one $p\in\pb\gD\bsb$ because the monodromy group $G_{\proj(q)}(\gl\of)$ with base point $\proj(q)$ is related to that with base point $\proj(p)$ by
\begin{equation}
\gG_p^{~q}(\gl\pb\of)G_{\proj(q)}(\gl\of)\gG_{q}^{~p}(\gl\pb\of)=G_{\proj(p)}(\gl\of).
\label{eq:behav_monodromy}
\end{equation}

\item Similarly, the pushforward of $\la f_{\gl,p}\ra$ to $\gD\bsb$ exists if and only if
\begin{equation}
\forall q,\ti q\in\pb \gD\bsb~~\text{with}~~\proj(q)=\proj(\ti q):~~~~~\la f_{\gl,p}(q)\ra=\la f_{\gl,p}(\ti q)\ra.
\label{eq:monodromy_cond_2}
\end{equation}
With $\pb\of:=\proj^*\of$ and using Def \ref{defi:trafos_surf} and \eqref{eq:propsgG2}, we may write
$$\la f_{\gl,p}(q)\ra=\gG_p^{~q}(\gl\pb\of)\la \pb f(q)\ra,~~~~\la f_{\gl,p}(\ti q)\ra=\gG_p^{~q}(\gl\pb\of)\gG_{q}^{~\ti q}(\gl\pb\of)\gG_{q}^{~p}(\gl\pb\of)\gG_{p}^{~q}(\gl\pb\of)    \la \pb f(\ti q)\ra.$$
Since $\la \pb f(q)\ra=\la \pb f(\ti q)\ra$ if $\proj(q)=\proj(\ti q)$, condition \eqref{eq:monodromy_cond_2} is thus equivalent to
\begin{equation}
\forall q,\ti q~~\text{with}~~\proj(q)=\proj(\ti q):~~~~~~~~~~\la f_{\gl,p}(q)\ra=\gG_p^{~q}(\gl\pb\of)\gG_{q}^{~\ti q}(\gl\pb\of)\gG_{q}^{~p}(\gl\pb\of) \la f_{\gl,p}(q)\ra.
\label{eq:monodromy_cond_3}
\end{equation}
Due to \eqref{eq:behav_monodromy}, condition \eqref{eq:monodromy_cond_3} is equivalent to
$$\forall q\in \pb \gD\bsb:~~~~~G_{\proj(p)}(\gl\of)\la f_{\gl,p}(q)\ra=\la f_{\gl,p}(q)\ra,$$
which is equivalent to the invariance of every point in the image of $\la f_{\gl,p}\ra$ under $\mon_{\proj(p)}(\gl\of)$.
\end{enumerate}
\end{proof}
%In Sects \ref{section:surf_pole_fo} and \ref{section:surf_pole_so}, we will compute the monodromy of $\gl\of$ and use Prop \ref{prop:monodromy_dar_cal_trafo} to relate the monodromy to the existence 
We note that as soon as there is one point $P\in \gD\bsb$ such that $\la f(P)\ra$ is not invariant under $\big(\mon_P(\gl\of)\big)^j$, the pushforward of no $\gl$-Calapso transform to the $j$-fold cover of $\gD\bsb$ exists.

The domains of the Darboux and Calapso transforms of $(\la \pb f\ra,\pb\Qua)$ are the universal cover $\pb\gD\bsb$ of $\gD\bsb$, which does not contain $\sip$. Therefore, in order to investigate whether the transforms have a limit at $\sip$, we first need to add the limit point $\sip$ to $\pb\gD\bsb$ to obtain the \textit{branched universal cover}.
\begin{defiandlemma}\label{defilamm:branchedcover}
Let $\gD$ be a simply connected Riemann surface and $\sip$ an interior point of $\gD$. Denote by $\pb \gD\bsb$ the universal cover of $\gD\bsb$. Then the union\footnote{We remind that we used the compound symbol $\pb \gD\bsb$ to denote the universal cover of $\gD\bsb$. We now define $\pb \gD$ as the union of that universal cover and $\{\sip\}$ such that indeed $\pb \gD=(\gD\bsb)\cup\{\sip\}$.} $(\pb\gD\bsb)\cup \{\sip\}=:\pb \gD$ can be equipped with a topology such that $\sip$ is a limit point of $\pb\gD\bsb$ in $\pb \gD$ and such that the subspace topology of $\pb\gD\bsb\subset\pb\gD$ agrees with the topology induced by $\proj:\pb\gD\bsb\rightarrow \gD\bsb$. The space $\pb \gD$ equipped with this topology is the {\upshape branched universal cover of} $\gD\bsb$.
\end{defiandlemma}
\begin{proof}
To construct the desired topology, choose a neighbourhood $\nei\subset \gD$ of $\sip$ and a holomorphic coordinate $z$ on $\nei$ such that $z(\nei)\subset\Com$ is the open unit disc with centre $0$. Denote by $\proj$ the covering map and by $\pb\nei\subset\pb \gD\bsb$ the preimage of $\nei\bsb$ under $\proj$. Let $(r,\phi)$ be polar coordinates on $\pb \nei$, that is, real smooth functions that satisfy $z\circ\proj=re^{i\phi}$. Then the map
\begin{equation}
h:\pb\nei\cup\{\sip\}\rightarrow \Com,~~~~~~~\pb\nei\ni p\mapsto r(p)e^{i\arctan(\phi(p))},~~~~~~~\sip\mapsto 0,
\end{equation}
is a bijection from $\pb\nei\cup \{\sip\}$ to a simply connected subset $\mathcal D$ of $\Com$. Now define a topology on $\pb \gD$ by saying that a subset $A\subset \pb \gD$ is open if and only if $A\cap (\pb\gD\bsb)$ is open in the topology induced by $\proj$ on $\pb\gD\bsb$ and $h\big(A\cap (\pb\nei\cub)\big)\subset \mathcal D$ is open in the subspace topology of $\mathcal D\subset\Com$. Clearly, this indeed defines a topology on $\pb \gD$ such that $\sip$ is a limit point of $\pb \gD\bsb$ in $\pb \gD$.
\end{proof}

By Lemma \ref{lemma:closedness_of_oneform}, the connections $\dd+\gl\of$ associated to a holomorphically isothermic surface with domain $M$ are flat on all of $M$, including the zeros of $\of$. Therefore, if additionally $M$ is simply connected, all primitives $\gG_p(\gl\of)$ are defined on all of $M$. We thus get
\begin{corollary}
	Let $(\la f\ra,\Qua)$ be holomorphically isothermic with simply connected domain $\gD$ such that $\Qua$ has a zero at an interior point $\sip \in\gD$, but is nowhere zero otherwise. Then all Darboux and all Calapso transforms of the pullback of $(\la f\ra,\Qua)$ to the universal cover of $\gD\bsb$ can be pushed forward to $\gD\bsb$. Moreover, all Darboux and all Calapso transforms have continuous limits at $\sip$.
\end{corollary}
Hence, zeros of $\Qua$ are conceptually much simpler than poles. However, this does not mean that the transformation theory of simple isothermic surfaces can readily be extended to simply connected holomorphically isothermic surfaces. In particular, Darboux transforms of holomorphically isothermic surfaces do not immerse at the zeros of $\Qua$ and hence are not again holomorphically isothermic in the sense of Def \ref{defi:iso_surf} (cf. \cite[Ch 2]{fu18}).

%----------------------------------------------------------------
%--- L I M I T S   O F   P R I M I T I V E S --------------------
%----------------------------------------------------------------

\section{Monodromy and limits of primitives of pole forms on 2-dimensional discs}\label{section:limits_of_prims_surf}
Throughout this section, $\gD$ denotes a compact Riemann surface diffeomorphic to a closed disc and $\sip$ an interior point of $\gD$. Due to \eqref{eq:propsgG2}, in order to compute the monodromy and investigate the limiting behaviour of primitives of a $\pro\ort(\Rmn)$-valued 1-form on $\gD$ around and at $\sip$, we may always replace $\gD$ by a smaller closed neighbourhood of $\sip$. Using such a replacement, we can achieve that a holomorphic coordinate defined in a neighbourhood of $\sip$ is defined on all of $\gD$. This is not necessary, but convenient. By $\pb\gD\bsb$ and $\pb\gD$, we denote the universal and branched universal covers of $\gD\bsb$, respectively, with covering map $\proj$, as defined in Def \ref{defilamm:branchedcover}.

We continue the notation of Sect \ref{section:stateofart}. A bold symbol always denotes a quantity defined on $\pb\gD\bsb$ which is the pullback via $\proj$ of some quantity with domain $\gD\bsb$ or $\gD$ and denoted by the corresponding normal symbol. Moreover, the capital greek letters $\Of$, $\Ppf$ and $\Pf$ always denote the orthogonal lifts of $\pro\ort(\Rmn)$-valued 1-forms $\of$, $\ppf$ and $\pf$, respectively.
\subsection{Preliminaries}
Let $z$ be a holomorphic coordinate on $\gD$. On $\pb\gD\bsb$, we define \textit{polar coordinates}
$$r:\pb \gD\bsb\rightarrow (0,\infty),~~~\phi:\pb \gD\bsb\rightarrow \R,$$
associated to $z$ up to translations $\phi\mapsto \phi+2\pi j$ with $j\in\mathbb Z$ by
$$re^{i\phi}=\pb z-z(\sip).$$
For such polar coordinates, let $r_0$ be the radius of the largest disc with centre $z(\sip)$ that is entirely contained in $z(\gD)$. For $\vphi\in\R$, we denote by $\pb\mu_\vphi$ the \textit{$r$-parameter line at angle $\vphi$}, that is,
$$\pb \mu_\vphi:(0,r_0]\rightarrow \pb\gD\bsb,~~~~T\mapsto \pb\mu_\vphi(T),~~~~~~\forall T\in(0,r_0]:~~r(\pb\mu_\vphi(T) )=T,~~\phi(\pb\mu_\vphi(T))=\vphi.$$
Similarly, for $T\in(0,r_0]$, we denote by $\pb \gga_T$ the \textit{$\phi$-parameter circle at radius $T$}, that is,
$$\pb \gga_T:\R\rightarrow \pb\gD\bsb,~~~~\vphi\mapsto \pb\gga_T(\vphi),~~~~~\forall \vphi\in\R:~~r(\pb\gga_T(\vphi) )=T,~~\phi(\pb\gga_T(\vphi))=\vphi.$$
By $\mu_\vphi:=\proj\circ \pb\mu_\vphi$ and $\gga_T:=\proj\circ\pb\gga_T$ we denote the corresponding curves in $\gD\bsb$. 

We would like to characterize, when a map with values in $O(\Rmn)$ or the Lie algebras $\ort(\Rmn)$ or $\pro\ort(\Rmn)$ tends to infinity. To do this, we view $O(\Rmn)$ and $\ort(\Rmn)$ as subsets of $\End(\Rmn)$ and equip $\End(\Rmn)\simeq \R^{(n+2)^2}$ with a positive definite, submultiplicative norm $|\cdot|$ such that $|A|=|A^*|$ for the adjoint $A^*$ of $A\in \End(\Rmn)$ with respect to the Minkowski inner product. On $\pro\ort(\Rmn)$, we define $|\cdot|$ by the pushforward of $|\cdot|$ restricted to $\ort(\Rmn)$ by the isomorphism \eqref{eq:iso_liealg}. We also denote by $|\cdot|$ a positive definite norm on $\Rmn$ such that $|Av|<|A||v|$ for all $A\in\End(\Rmn)$ and all $v\in\Rmn$.

Let $\ga$ be a 1-form on $\gD\bsb$ with values in some vector space with norm $|\cdot|$ and $z=u+iv$ a holomorphic coordinate on $\gD$. We say that $\ga$ is bounded if its component functions $\ga_u$ and $\ga_v$ with respect to $\dd u$ and $\dd v$ are bounded. Since $\gD$ is compact, this is independent of the chosen coordinate $z$ on $\gD$. In terms of polar coordinates $(r,\phi)$ associated to $z$, the pullback $\pb\ga$ can then be written as
\begin{equation}
\pb\ga=\big(\cos\phi~(\ga_u\circ\proj)+\sin\phi~(\ga_v\circ\proj)\big)\dd r+r\big(-\sin\phi~(\ga_u\circ\proj)+\cos\phi~(\ga_v\circ\proj)\big)\dd\phi.
\end{equation}
Therefore, if $\ga$ is bounded, then for every holomorphic coordinate $z$ on $\gD$ there is a constant $B\in\R$ such that
\begin{equation}
	\forall T\in (0,r_0]:~~~|\gga_T^*\ga|=|\pb\gga_T^*\pb\ga|=\left|\pb\gga_T^*\left(r\big(-\sin\phi~(\ga_u\circ\proj)+\cos\phi~(\ga_v\circ\proj)\big)\dd\phi\right)\right|<T~B~|\dd\vphi|.
	\label{eq:linear_bound_for_angbound}
\end{equation}

% ------------------  P R I M I T I V E S ---------------------------------
\subsection{Primitives of pure pole forms}\label{sec:pure_pole_forms_discs}
We now introduce the notion of a coordinate adapted to a meromorphic 1-form on $\gD$ with a pole of first order at $\sip$.
\begin{defiandlemma}\label{lemma:coordinate_adapted_to_1form}
Let $\ga$ be a nowhere zero holomorphic 1-form on $\gD\bsb$ with a pole of first order at $\sip$. Then there is a neighbourhood $\nei\subset\gD$ of $\sip$ and a holomorphic coordinate $z$ on $\nei$ such that
$$\ga|_\nei=\Res_{\sip}(\ga)\frac{\dd z}{z},$$
where $\Res_{\sip}(\ga)$ is the residue of $\ga$ at $\sip$. We call $z$ a {\upshape coordinate adapted to} $\ga$. Such a coordinate is unique up to multiplication by a complex constant.
\end{defiandlemma}
\begin{proof}
Think of $p\in\gD\bsb$ fixed. It follows from the residue theorem that the function
\begin{equation}
z:\gD\bsb\rightarrow \Com,~~~q\mapsto e^{\big(\Res_{\sip}(\ga)\big)^{-1}\int_p^q \ga},\label{eq:expr_holcoord}
\end{equation}
is well defined, that is, independent of the path from $p$ to $q$ along which the integral is computed. But $z$ has a holomorphic extension to all of $\gD$. To see this, choose a holomorphic coordinate $Z$ on a neighbourhood $\ti \nei$ of $\sip$ and write
$$\frac{\ga}{\Res_{\sip}(\ga)}=\frac{\dd Z}{Z-Z(\sip)}+H$$
with a holomorphic function $H$ on $\ti \nei$. Substituting this into \eqref{eq:expr_holcoord} shows that indeed $z$ is holomorphic on $\ti\nei$. Clearly, $\Res_{\sip}(\ga)\,\dd z/z=\ga$ on $\ti \nei$. In particular its derivative is nowhere zero and so there is a neighbourhood $\nei$ of $\sip$ where $z$ is bijective and thereby a holomorphic coordinate.

Now suppose, that there is another such coordinate $\ti z$. Then $\dd\ti z/\ti z=\dd z/z$, which implies that $\ti z/z$ is a complex constant.

\end{proof}

Analogously to \cite[Sect 3.1]{fu18p}, we now define pure pole forms on $\gD\bsb$.
\begin{defi}\label{defI:pure_pole_forms_s}
A $\pro\ort(\Rmn)$-valued 1-form $\ppf$ on $\gD\bsb$ is called a {\upshape pure pole form} if 
$$\ppf=\Re\left((\ppfr+i\ppfi)\ga\right),$$
where $\ga$ is a meromorphic 1-form with a pole of first order at $\sip$ and $\Res_{\sip}(\ga)=-1$, but is holomorphic and nowhere zero on $\gD\bsb$, and 
\begin{equation}
\ppfr=v\wedge w+\R\,\id,~~~~\ppfi=x\wedge y+\R\,\id~~~~~\text{with}~~~~~\la x,y\ra\subseteq \la v,w\ra^\perp\subset \Rmn.
\label{eq:orthogonality_pure_pole_form_s}
\end{equation}
In particular, $\ppfr$ and $\ppfi$ commute. For convenience, we assume that there is a coordinate $z$ adapted to $\ga$ defined on all of $\gD$, which we call a {\upshape coordinate adapted to $\ppf$}.

We say that a pure pole form $\ppf$ is {\upshape Minkowski, spacelike} or {\upshape degenerate} according to the signature of $\la v,w\ra$. In the Minkowski case, denote by $\zeta$ the positive eigenvalue of $v\wedge w$. If $\ppf$ is Minkowski with $\zeta<1$, spacelike or degenerate, we say that $\ppf$ is {\upshape of the first kind}. Otherwise, it is {\upshape of the second kind}. 
%Moreover, if the absolute values of the eigenvalues of $v\wedge w$ and $x\wedge y$ coincide, we say that $\ppf$ is {\upshape isothermic}.
\end{defi}
The condition $\Res_s(\ga)=-1$ fixes the scaling of $\ppfr+i\ppfi$. The requirement that a globally defined coordinate adapted to $\ppf$ exist can always be satisfied by using the freedom to replace $\gD$ by a simply connected, closed neighbourhood of $\sip$.

In terms of polar coordinates $(r,\phi)$ associated to a coordinate $z$ adapted to $\ppf$, we have
\begin{equation}
\pb \ppf=\proj^*\ppf=\proj^*\Rea(-(\ppfr+i\ppfi)\dd z/z)=-\ppfr\frac{\dd r}{r}+\ppfi\dd\phi.
\label{eq:canonical_form_xi}
\end{equation}
In particular, $\dd+\ppf$ and $\dd+\pb\ppf$ are flat. Moreover, for any $r$-parameter line $\pb\mu_\vphi:(0,r_0]\rightarrow \pb\gD\bsb$, the pullback $\pb\mu^*_\vphi\pb\ppf$ is a pure pole form on $(0,r_0]$ in the sense of \cite[Sect 3.1]{fu18p} which is Minkowski, spacelike, degenerate, of the first kind or of the second kind if and only if $\ppf$ has that property.

Since $\ppfr$ and $\ppfi$ commute, we find that
\begin{equation}
\gG_p(\pb \ppf)=\gG_p(\ppfi\dd\phi)\gG_p(-\ppfr\dd r/r)=\gG_p(-\ppfr\dd r/r)\gG_p(\ppfi\dd\phi)=e^{\ln\left(\frac{r(p)}{r}\right)\ppfr}e^{(\phi-\phi(p))\ppfi},
\label{eq:factori_xi}
\end{equation}
where $e:\pro\ort(\Rmn)\rightarrow \Pro O(\Rmn)$ is the exponential map. Due to this commutativity, the monodromy of a pure pole form is independent of the base point and its radial component $\ppfr$ and given by the simple exponential
\begin{equation}
\mon(\ppf)=e^{-2\pi\ppfi}.
\label{eq:monodromy_purepoleform}
\end{equation}
Moreover, the angular component $\ppfi\dd\phi$ is bounded and only its radial component is singular at $\sip$. Therefore, for any compact subset $\comp\subset \R$, there is a constant $B\in\R$ such that
\begin{equation}
\forall \vphi_1,\vphi_2\in\comp:~~~~~|\gG_{\vphi_1}^{~\vphi_2}(\Ppfi\dd\vphi)|<B.
\label{eq:boundedness_of_angpart_purepoleform_curve}
\end{equation}

We now investigate the limiting behaviour of primitives of pure pole forms.
\begin{lemma}\label{lemma:primitives_purepoleforms_s}
Let $\ppf=\Rea(-(\ppfr+i\ppfi)\dd z/z)$ be a pure pole form on $\gD\bsb$. For any subset $\comp\subset\pb\gD\bsb$ such that $\comp\cub\subset\pb\gD$ is compact with limit point $\sip$, there is a $B\in\R$ such that
\begin{enumerate}
	\item if $\ppf$ is Minkowski,
	\begin{align}
\forall p,q\in\comp~\text{with}~|\pb z(q)|\leq |\pb z(p)|:~~~~	&\left|\left|\frac{\pb z(q)}{\pb z(p)}\right|^\zeta\gG_{p}^{~q}(\pb \Ppf)\right|<B,\label{eq:upbound_ex_mink_s}\\
	\forall p\in\comp:~~~~&\lqbc\left|\frac{\pb z(q)}{\pb z(p)}\right|^\zeta\gG_{p}^{~q}(\pb \Ppf)=\frac{V_+V_-^*}{\ipl V_+,V_-\ipr},\label{eq:lim_ex_mink_s}
	\end{align}	
	where $V_\pm$ are eigenvectors of $\Ppfr$ with eigenvalues $\pm\zeta$, $\zeta>0$.
	\item If $\ppf$ is spacelike,
	\begin{equation}
\forall p,q\in \comp:~~~~~\left|\gG_{p}^{~q}(\pb \Ppf)\right|<B,
	\label{eq:upbound_ex_spacelike_s}
	\end{equation}
	but $\gG_{p}(\pb \ppf)$, restricted to $\comp$, does not have a limit at $\sip$.
	\item If $\ppf$ is degenerate, let the symmetric function $\bd$ be given by
	\begin{align}
	\bd:(\pb\gD\bsb)^2&\rightarrow \R,\notag\\
	(p,q)&\mapsto 1+[\ln |\pb z(p)/\pb z(q)|]^2.\label{eq:defi_bounding_function_def_s}
	\end{align}
	Then
	\begin{align}
\forall p, q\in \comp:~~~~	&\left|\frac{1}{\bd(q,p)}\gG_{p}^{~q}(\pb \Ppf)\right|<B,
	\label{eq:upbound_ex_degen_s}\\
\forall p\in\comp:~~~~&\lqbc\frac{1}{\bd(q,p)}\gG_{p}^{~q}(\pb \Ppf)=-\frac{\|W\|^2}{2}V_0V_0^*,\label{eq:lim_ex_deg_s}
	\end{align}
where $\Ppfr=V_0\wedge W$ with $\|V_0\|^2=0$. 
\end{enumerate}
%If $\ppf$ is of the first kind and isothermic, \eqref{eq:upbound_ex_mink_s} to \eqref{eq:lim_ex_deg_s} hold without the requirement on $\comp\subset\pb\gD\bsb$ to be compact.
\end{lemma}
\begin{proof}
The bounds \eqref{eq:upbound_ex_mink_s}, \eqref{eq:upbound_ex_spacelike_s} and \eqref{eq:upbound_ex_degen_s} follow from \eqref{eq:factori_xi}, the bound \eqref{eq:boundedness_of_angpart_purepoleform_curve} and \cite[Prop 3.1]{fu18p}. The limits \eqref{eq:lim_ex_mink_s} and \eqref{eq:lim_ex_deg_s} follow from \eqref{eq:factori_xi}, \cite[Prop 1]{fu18p}, \eqref{eq:orthogonality_pure_pole_form_s} and again the bound \eqref{eq:boundedness_of_angpart_purepoleform_curve}. Finally, if the restriction of $\gG_p(\pb\ppf)$ to $\comp$ had a limit at $s$, then also its restriction to some segment of some $r$-parameter line contained in $\comp$ and ending in $s$ would have a limit at $s$, which would contradict \cite[Prop 3.1]{fu18p}.

%To see that in the spacelike case $\gG_{p}(\pb \ppf)$ does not have a limit at $\sip$, let $v,w\in\Rmn$ be such that $\Ppfr=v\wedge w$. Then, using \eqref{eq:orthogonality_pure_pole_form_s}, for $V\in\la v,w\ra$ we get
%$$\gG_p(\pb\ppf)\la V\ra=\gG_p(-\ppfr \dd r/r)\la V\ra,$$
%which does not have a limit at $\sip$ because $\gG_p(-\ppfr \dd r/r)$ is a rotation around the $n$-dimensional axis $\la v,w\ra^\perp$ with speed increasing towards infinity as $\sip$ is approached. Therefore, $\gG_p(\pb\ppf)$ cannot have a limit at $\sip$.
\end{proof}
We note that the functions $\left|\frac{\pb z(q)}{\pb z(p)}\right|^\zeta$ and $\bd(p,q)$ are independent of the chosen coordinate adapted to $\ppf$.

%---------------------- P O L E  F O R M S ----------------------------
\subsection{Primitives of pole forms}\label{section:poleforms_on_discs}
We now study the behaviour of primitives of those $\pro\ort(\Rmn)$-valued \mbox{1-forms} $\pb\pf$ for which $\pf$ is a sum of a pure pole form and a bounded $\pro\ort(\Rmn)$-valued 1-form. 
\begin{defi}
A $\pro\ort(\Rmn)$-valued 1-form $\pf$ on $\gD\bsb$ is a {\upshape pole form} if $\dd+\pf$ is flat and $\pf-\ppf$ is bounded for a pure pole form $\ppf$ on $\gD\bsb$.
\end{defi}
Let $(r,\phi)$ be polar coordinates associated to a coordinate $z$ adapted to $\ppf$. Then the pullback of $\ppf$ to any $r$-parameter line is a pure pole form in the sense of \cite[Sect 3.1]{fu18p} such that the pullback of $\pf$ is a pole form in the sense of \cite[Def 3]{fu18p}.

 Thus, we know from Cor 1 and Prop 2 of \cite{fu18p} how the primitives of $\pf$ behave as one approaches $s$ along an $r$-parameter line. We now investigate how the primitives of the pullbacks of $\pf$ via the family $(\gga_T)_{T\in(0,r_0]}$ of $\phi$-parameter circles behave as $T$ tends to zero. It follows directly from \eqref{eq:linear_bound_for_angbound} that there is a constant $B\in\R$ such that
\begin{equation}
\forall T\in(0,r_0]:~~~~~|\gga_T^*\pf-\ppfi\dd \vphi|<T\,B\,|\dd\vphi|.
\label{eq:bound_angcomp_poleform}
\end{equation}
We now use this bound to get a corresponding bound for the primitives of $\gga^*_T\pf$.
\begin{lemma}\label{lemma:uniform_convergence_angularpart}
Let $\pf$ be a pole form on $\gD\bsb$ with $\pf-\ppf$ bounded for the pure pole form $\ppf$. Choose a coordinate $z$ adapted to $\ppf$ with associated polar coordinates $(r,\phi)$ and denote by $r_0$ the radius of the largest disc with centre $z(\sip)$ in $z(\gD)$. Let $(\gga_T)_{T\in(0,r_0]}$ be the family of $\phi$-parameter circles in $\gD\bsb$. Then for every compact subset $\comp\subset\R$, there is a constant $B\in\R$ such that
\begin{equation}
\forall T\in(0,r_0] ~~\forall \vphi_1,\vphi_2\in\comp: ~~~~~~|\gG_{\vphi_1}^{~\vphi_2}(\gga_T^*\Pf)|<B,
\label{eq:uniform_bound_angpart_primpoleform}
\end{equation}
and
\begin{equation}
\forall T\in(0,r_0]~~\forall \vphi_1,\vphi_2\in\comp:~~~~~~|\gG_{\vphi_1}^{~\vphi_2}(\gga_T^*\Pf)-\gG_{\vphi_1}^{~\vphi_2}(\Ppfi\dd\vphi)|<B~T.
\label{eq:uniform_cont_angpart_poleform}
\end{equation}
\end{lemma}
\begin{proof}
We first show \eqref{eq:uniform_bound_angpart_primpoleform}. As in \cite[Prop A.1]{fu17}, we have the bound
$$|\gG_{\vphi_1}^{~\vphi_2}(\gga_T^*\Pf)|\leq e^{\left|\int_{\vphi_1}^{\vphi_2}|\gga_T^*\Pf|\right|}\leq e^{\left|\int_{\vphi_1}^{\vphi_2}|\gga_T^*\Pf-\Ppfi\dd\vphi|\right|}e^{\left|\int_{\vphi_1}^{\vphi_2}|\Ppfi|\dd\vphi\right|}.$$
This, the boundedness of $\Pf-\Ppf$ and \eqref{eq:boundedness_of_angpart_purepoleform_curve} show that indeed there is a $ B\in\R$ such that \eqref{eq:uniform_bound_angpart_primpoleform} holds. To show \eqref{eq:uniform_cont_angpart_poleform}, we note that by differentiation and setting $\vphi_1=\vphi_2$ one readily verifies that
\begin{equation}
\gG_{\vphi_1}^{~\vphi_2}(\gga_T^*\Pf)-\gG_{\vphi_1}^{~\vphi_2}(\Ppfi\dd\vphi)=\int_{\vphi_1}^{\vphi_2}\gG_{\vphi_1}^{~\ti\vphi}(\gga^*_T\Pf)\left((\gga_T^*\Pf)_{\ti \vphi}-\Ppfi\right)\gG_{\ti\vphi}^{~\vphi_2}(\Ppfi\dd\vphi)\dd\ti\vphi
\label{eq:equationfory_ang}
\end{equation}
holds for all $T\in(0,r_0]$ and all $\vphi_1,\vphi_2\in\R$. Now take the norm $|\cdot|$ of this and use the bounds \eqref{eq:bound_angcomp_poleform}, \eqref{eq:uniform_bound_angpart_primpoleform} and \eqref{eq:boundedness_of_angpart_purepoleform_curve} to find that indeed there is a constant $B$ such that \eqref{eq:uniform_cont_angpart_poleform} holds.

\end{proof}
We now come to the main result of this section for pole forms $\pf$ with $\pf-\ppf$ bounded for a pure pole form $\ppf$ of the first kind. 
\begin{prop}\label{prop:maintool_firstkind_surf}
Let $\pf$ be a pole form on $\gD\bsb$ with $\pf-\ppf$ bounded for a pure pole form $\ppf$ of the first kind. For every $p\in\pb\gD\bsb$ denote the gauge transform of $\pb\pf$ using the map $\gG^{~p}(\pb \ppf)$ by
\begin{equation}
\Gtr {\pb \ppf} p {\pb \pf}:=\gG_{p}(\pb \ppf)\pb \pf\gG^{~p}(\pb \ppf)-\pb \ppf,
\label{eq:defi_gaugetrafo_firstkind}
\end{equation}
such that
\begin{equation}
\forall p\in\pb \gD\bsb:~~\gG_{p}(\pb \pf)=\gG_{p}(\Gtr {\pb \ppf} {p} {\pb \pf})\gG_{p}(\pb \ppf)
\label{eq:factorization_into_surf}
\end{equation}
and
\begin{equation}
\forall p,\ti p\in\pb\gD\bsb:~~~\gG_{\ti p}(\Gtr {\pb \ppf} {\ti p} {\pb \pf})=\gG_{\ti p}^{~p}(\pb\pf)\gG_p(\Gtr {\pb \ppf} {p} {\pb \pf})\gG_{p}^{~\ti p}(\pb\ppf).
\label{eq:trafobehav_prim_gaugetrafo}
\end{equation}

\begin{enumerate}
	\item There is a limit map $p\mapsto\gG_p^{~\sip}(\Gtr {\pb \ppf} p {\pb \pf})$ defined on all of $\pb\gD\bsb$ such that for all $\comp\subset \pb\gD\bsb$ for which $\comp\cub\subset \pb\gD$ is compact with limit point $\sip$ we have
	\begin{equation}
	\forall p\in\pb\gD\bsb:~~~~~\lqbc\gG_p^{~q}(\Gtr {\pb \ppf} p {\pb \pf})=\gG_p^{~\sip}(\Gtr {\pb \ppf} p {\pb \pf}).
	\label{eq:limit_gaugetransformedconnectino_fk_s}
	\end{equation}

	Moreover, if $\ppf$ is degenerate or spacelike, that limit map satisfies
	\begin{equation}
	\lim_{p\rightarrow \sip} \gG_{p}^{~\sip}(\Gtr {\pb \ppf} {p} {\pb \pf})=\id.
	\label{eq:convergence_primofomegati}
	\end{equation}

	\item The monodromy of $\pf$ with base point $P\in\gD\bsb$ is
	\begin{equation}
	\mon_{P}(\pf)=\gG_{p}^{~\sip}(\Gtr {\pb \ppf} {p} {\pb \pf})\,\mon(\ppf)\,\gG_s^{~p}(\Gtr {\pb \ppf} {p} {\pb \pf}),
	\label{eq:monodromy_first_kind}
	\end{equation}
	where $p$ is an arbitrary point in $\pb\gD\bsb$ such that $\proj(p)=P$ and $\gG_{p}^{~\sip}(\Gtr {\pb \ppf} {p} {\pb \pf})$ is the above limit map.
\end{enumerate}
\end{prop}
\begin{proof}
Again, due to \eqref{eq:propsgG2}, we may replace $\gD$ by any closed, simply connected neighbourhood of $\sip$. By doing so, we can achieve that there is a coordinate $z$ adapted to $\ppf$ that is defined on all of $\gD$ and such that $z(\gD)\subset\Com$ is a closed disc with some radius $r_0>0$. By $(r,\phi)$, we denote polar coordinates associated to $z$.

The relations \eqref{eq:factorization_into_surf} and \eqref{eq:trafobehav_prim_gaugetrafo} follow from \eqref{eq:propsgG2} and \eqref{eq:gauge_trafo_beha}.
\begin{enumerate}
	\item It is sufficient to prove the proposition for all $p$ and $\comp$ such that $p\in\comp$. Namely, if \eqref{eq:limit_gaugetransformedconnectino_fk_s} holds for some $\comp_1\subset \pb\gD\bsb$ such that $\comp_1\cub$ is compact with limit point $\sip$, it also holds for all $\comp_2\subseteq\comp_1$ with limit point $\sip$. Therefore, the cases in which $p\notin\comp$ follow from the cases in which $p\in\comp$. So fix $\comp\subset\pb\gD\bsb$ such that $\comp\cub\subset\pb\gD$ is compact with limit point $\sip$ and a point $p\in\comp$. Denote by $\comp_\R\in\R$ any compact subset of $\R$ such that $\phi(\comp)\subseteq \comp_\R$, where $(r,\phi)$ are polar coordinates associated to $z$. Since we assumed $z(\gD)$ to be a disc, we can connect $p$ to a point $q\in\comp$ by first moving radially along the $r$-parameter line $\pb \mu_{\phi(p)}$ from $r(p)$ to $r(q)$ and then moving along the $\phi$-parameter curve $\pb \gga_{r(q)}$ from $\phi(p)$ to $\phi(q)$. To simplify notation, we just write $\pb\mu$ for $\pb \mu_{\phi(p)}$ throughout this proof. Accordingly
\begin{equation}
\gG_p^{~q}(\Gtr {\pb \ppf} {p} {\pb \pf})=\gG_{r(p)}^{~r(q)}(\pb \mu^*(\Gtr {\pb \ppf} {p} {\pb \pf}))\gG_{\phi(p)}^{~\phi(q)}(\pb \gga^*_{r(q)}(\Gtr {\pb \ppf} {p} {\pb \pf})).
\label{eq:continuity_first_kind_surf}
\end{equation}
To prove that the left hand side has a limit at $\sip$, we show that both factors on the right hand side have a limit at $\sip$. 

To see that the first factor has a limit at $\sip$, we note that, since $p=\mu(r(p))$, we have
$$\pb\mu^*(\Gtr {\pb \ppf} {p} {\pb \pf})=\gG_{r(p)}(\mu^*\ppf)(\mu^*\pf)\gG^{~r(p)}(\mu^*\ppf)-\mu^*\ppf.$$
This is precisely the gauge transform of the pole form $\mu^*\pf$ in the sense of \cite[Def 3]{fu18p} on $(0,r_0]$ by the map $\gG^{~r(p)}(\mu^*\ppf)$, where $\mu^*\ppf$ is a pure pole form of the first kind such that $\mu^*\pf-\mu^*\ppf$ is bounded. We can therefore conclude from \cite[Cor 1]{fu18p} that the first factor on the right hand side of \eqref{eq:continuity_first_kind_surf} has a limit at $\sip$.

We now prove that the second factor in \eqref{eq:continuity_first_kind_surf} converges to the identity as $q$ approaches $\sip$ in $\comp$. For $T\in(0,r_0]$ we have
\begin{equation}
\pb\gga^*_{T}(\Gtr {\pb \ppf} {p} {\pb \pf})=\gG_{r(p)}^{~T}(\mu^*\ppf)~\gG_{\phi(p)}(\ppfi\dd\vphi)\gga^*_{T}(\pf-\ppf) \gG^{~\phi(p)}(\ppfi\dd\vphi) ~\gG^{~r(p)}_{T}(\mu^*\ppf),
\label{eq:contifirstkindsurftemp1}
\end{equation}
where we used 
$$\pb\gga^*_T\gG_p(\pb\ppf)=\gG_{r(p)}^{~T}(\mu^*\ppf)~\gG_{\phi(p)}(\ppfi\dd\vphi).$$
By \eqref{eq:boundedness_of_angpart_purepoleform_curve}, $\gG_{\phi(p)}^{~\phi(q)}(\ppfi\dd\vphi)$ is bounded by some constant for all $q\in\comp$. Since $\pf$ is a pole form, $\pf-\ppf$ is bounded and so by \eqref{eq:linear_bound_for_angbound}, there is a constant $C_1\in\R$ such that on $\comp_\R$
$$\forall T\in(0,r_0]:~~~~|\gG_{\phi(p)}(\ppfi\dd\vphi)\gga^*_{T}(\pf-\ppf) \gG^{~\phi(p)}(\ppfi\dd\vphi)|<  C_1~T|\dd\vphi|.$$
Thus, the 1-form $\pb\gga^*_{T}(\Gtr {\pb \ppf} {p} {\pb \pf})$ in  \eqref{eq:contifirstkindsurftemp1} is of the form
\begin{equation}
\pb\gga^*_{T}(\Gtr {\pb \ppf} {p} {\pb \pf})=T~\gG_{r(p)}^{~T}(\mu^*\ppf)\,\chi(T,\cdot)\, \gG^{~r(p)}_{T}(\mu^*\ppf)\,\dd\vphi,\\
\label{eq:contifirstkindsurftemp2}
\end{equation}
where $\chi$ is a continuous map from $(0,r_0]\times \R$ to $\pro\ort(\Rmn)$ which is bounded on $(0,r_0]\times\comp_\R$. One can easily generalize \cite[Lemma 1]{fu18p} to bounded maps $\chi$ with domain $(0,r_0]\times\comp_\R$ (see \cite[Lemma 4.1.12]{fu18}) to conclude that there is family $(\cB_\rho)_{\rho\in(0,r_0]}$ of integrable functions on $(0,r_0]$ such that 
\begin{equation}
\forall \rho\in(0,r_0]:~~~~\lim_{T\rightarrow 0}T\,\cB_\rho(T)=0
\label{eq:prop_of_integrablefunctionbound}
\end{equation}
and
$$\forall \rho,T\in(0,r_0]~\forall \vphi\in\comp_\R:~~~~|\gG_{\rho}^{~T}(\mu^*\ppf)\chi(T,\vphi) \gG^{~\rho}_{T}(\mu^*\ppf)|< \cB_\rho(T).$$
Using this, the norm of \eqref{eq:contifirstkindsurftemp2} is seen to satisfy
$$|\pb\gga^*_{T}(\Gtr {\pb\ppf} p {\pb\pf})|\leq T\,|\gG_{r(p)}^{~T}(\mu^*\ppf)\chi(T,\cdot) \gG^{~r(p)}_{T}(\mu^*\ppf)|\leq T\,\cB_{r(p)}(T).$$
Now use \cite[A.6]{fu17} to conclude that the second factor in \eqref{eq:continuity_first_kind_surf} satisfies
\begin{equation}
\left|\gG_{\phi(p)}^{~\vphi}(\pb\gga^*_T(\Gtr {\pb\ppf} p {\pb\pf}))-\id\right|\leq e^{\left|\int_{\phi(p)}^\vphi TB_{r(p)}(T)\dd\ti\vphi  \right|  }-1=e^{\left|(\vphi-\phi(p)) TB_{r(p)}(T)\right|  }-1
\label{eq:secondfactor}
\end{equation}
for all $\vphi\in\comp_\R$. Taking the limit $T\rightarrow 0$ of this using \eqref{eq:prop_of_integrablefunctionbound}, we finally find that indeed 
$$\forall \vphi\in\comp_\R:~~~~\lim_{T\rightarrow 0}\gG_{\phi(p)}^{~\vphi}(\pb\gga^*_T(\Gtr {\pb\ppf} p {\pb\pf}))=\id.$$
Thus, the second factor in \eqref{eq:continuity_first_kind_surf} converges to the identity as one approaches $\sip$ in $\comp$. 

Taking the limit of \eqref{eq:continuity_first_kind_surf} thus yields
\begin{equation}
\lim_{q\rightarrow \sip}\gG_p^{~q}(\Gtr {\pb\ppf} p {\pb\pf})=\gG_{r(p)}^{~0}(\pb\mu^*\Gtr {\pb\ppf} p {\pb\pf})=:\gG_p^{~\sip}(\Gtr {\pb\ppf} p {\pb\pf}),
\label{eq:limfirstkind}
\end{equation}
which defines the limit map $p\mapsto \gG_p^{~\sip}(\Gtr {\pb\ppf} p {\pb\pf})$. Clearly, that map is independent of the chosen subset $\comp$. \eqref{eq:convergence_primofomegati} follows directly from \eqref{eq:limfirstkind} and \cite[Lemma 1]{fu18p}.

\item We now compute the monodromy of $\pf$ with base point $P$. Choose $p\in\pb\gD\bsb$ such that $\proj(p)=P$. The coordinate $z$ adapted to $\ppf$ is unique up to multiplication by a complex constant (see Lemma \ref{lemma:coordinate_adapted_to_1form}). We may thus conveniently choose $z$ and $\phi$ such that $\phi(p)=0$. Then
\begin{equation}
\mon_{P}(\pf)=\gG_{0}^{~2\pi}(\pb\gga^*_{r(p)}\pb\pf)=\gG_{0}^{~2\pi}(\pb\gga^*_{r(p)}\Gtr {\pb\ppf} p {\pb\pf})\mon(\ppf).
\label{eq:monodromy_firstkind_temp1}
\end{equation}
The 1-form $\Gtr {\pb\ppf} p {\pb\pf}$ is a gauge transform of $\pb\pf$ and hence $\dd+\Gtr {\pb\ppf} p {\pb\pf}$ is flat. Thus, instead of computing the primitive of $\Gtr {\pb\ppf} p {\pb\pf}$ along the $\phi$-parameter circle at radius $r(p)$, we may first integrate radially from $r(p)$ to some $T$, then along the $\phi$-parameter circle at radius $T$ and then radially back from $T$ to $r(p)$. Accordingly,
\begin{equation}
\gG_{0}^{~2\pi}(\pb\gga^*_{r(p)}\Gtr {\pb\ppf} p {\pb\pf})=\gG_{r(p)}^{~T}(\pb\mu^*_{0}\,\Gtr {\pb\ppf} p {\pb\pf})\gG_{0}^{~2\pi}(\pb\gga^*_{T}\,\Gtr {\pb\ppf} p {\pb\pf})\gG_{T}^{~r(p)}(\pb\mu^*_{2\pi}\,\Gtr {\pb\ppf} p {\pb\pf}).
\label{eq:monodromy_firstkind_temp_calc1}
\end{equation}
for arbitrary $T\in(0,r_0]$. We want to take the limit $T\rightarrow 0$. Substituting
$$\pb\mu^*_{2\pi}\,\Gtr {\pb\ppf} p {\pb\pf}=\mon(\ppf)\,\pb\mu^*_{0}\,\Gtr {\pb\ppf} p {\pb\pf}\,\mon(\ppf)^{-1}$$
into \eqref{eq:monodromy_firstkind_temp_calc1}, using \eqref{eq:gauge_trafo_beha} with a constant $g=\mon(\ppf)$, and substituting the result into \eqref{eq:monodromy_firstkind_temp1} yields
\begin{equation}
\mon_{P}(\pf)=\gG_{r(p)}^{~T}(\pb\mu^*_{0}\,\Gtr {\pb\ppf} p {\pb\pf})\gG_{0}^{~2\pi}(\pb\gga^*_{T}\,\Gtr {\pb\ppf} p {\pb\pf})\,\mon(\ppf)\,\gG_{T}^{~r(p)}(\pb\mu^*_{0}\,\Gtr {\pb\ppf} p {\pb\pf})
\label{eq:monodromy_firstkind_temp_calc2}
\end{equation}
for arbitrary $T\in(0,r_0]$. As $T$ tends to zero, the second factor on the right hand side tends to the identity, by \eqref{eq:secondfactor}. Thus, together with the previous result \eqref{eq:limfirstkind}, the limit $T\rightarrow 0$ of \eqref{eq:monodromy_firstkind_temp_calc2} yields \eqref{eq:monodromy_first_kind}. Using \eqref{eq:trafobehav_prim_gaugetrafo}, one may see that indeed the right hand side of \eqref{eq:monodromy_first_kind} is independent of the chosen $p$ with $\proj(p)=P$, as it has to be.
\end{enumerate}
\end{proof}

%-----------------------------------------------------------------------------------------------------

We now turn to primitives of pole forms $\pf$ where $\pf-\ppf$ is bounded for a Minkowski pure pole form $\ppf$.

\begin{prop}\label{prop:maintool_sing_s}
Let $\pf$ be a pole form on $\gD\bsb$ with $\pf-\ppf$ bounded for a Minkowski pure pole form $\ppf=\Rea(-(\ppfr+i\ppfi)\dd z/z)$. Denote by $V_\pm$ eigenvectors of $\Ppfr$ to the eigenvalues $\pm\zeta$, $\zeta>0$. Then there is a smooth map $K:\pb\gD\bsb\rightarrow \Li^{n+1}$ so that
\begin{equation}
\forall p,q\in\pb\gD\bsb:~~~~~K(p)=\left|\frac{\pb z(q)}{\pb z(p)}\right|^\zeta\gG_p^{~q}(\pb\Pf)K(q),~~~~~\lim_{\ti q\rightarrow \sip}K(\ti q)=\frac{V_+}{\ipl V_+,V_-\ipr}
\label{eq:prop_K_s}
\end{equation}
and, for every $\comp\subset\pb\gD\bsb$ such that $\comp\cub\subset\pb\gD$ is compact with limit point $\sip$, 
\begin{align}
\forall p\in\pb\gD\bsb:~~~~~\lqbc \left|\frac{\pb z(q)}{\pb z(p)}\right|^\zeta\gG_p^{~q}(\pb\Pf)&=K(p)V_-^*,\label{eq:limit_prim_mink_surf1}\\
\forall p\in\pb\gD\bsb:~~~~~\lqbc \left|\frac{\pb z(q)}{\pb z(p)}\right|^\zeta\gG_q^{~p}(\pb\Pf)&=V_- K(p)^*.\label{eq:limit_prim_mink_surf2}
\end{align}
\end{prop}
\begin{proof}
Again, due to \eqref{eq:propsgG2}, we may replace $\gD$ by an arbitrary simply connected, closed neighbourhood of $\sip$ and thereby achieve that $z(\gD)\subset\Com$ is a disc of some radius $r_0>0$. Let $(r,\phi)$ be polar coordinates associated $z$ and denote by $(\pb\mu_\vphi)_{\vphi\in\R}$ and $(\pb\gga_T)_{T\in(0,r_0]}$ the families of $r$- and $\phi$-parameter curves in $\pb\gD\bsb$, respectively. Similarly as in \cite[Prop 2]{fu18p}, define the map $K:\pb\gD\bsb\rightarrow \Li^{n+1}$ by $K(p)=\ti K(r(p),\phi(p))$, where
\begin{align*}
\ti K:(0,r_0]\times \R&\rightarrow \Li^{n+1},\\
(T,\vphi)&\mapsto \left[\int_{T}^{0}\left(\frac{\rho}{T}\right)^\zeta\gG_{T}^{~\rho}(\mu_{\vphi}^*\Pf)[(\mu_{\vphi}^*\Pf)_\rho-\Ppfr]\dd\rho +\id\right]\frac{V_+}{\ipl V_+,V_-\ipr}.
\end{align*}
Since $\mu^*_\vphi\pf$ is a pole form for all $\vphi\in\R$, it follows from \cite[Lemma 2]{fu18p} that $\left(\frac{\rho}{T}\right)^\zeta\gG_{T}^{~\rho}(\mu_{\vphi}^*\Pf)$ is bounded, such that the integral in $\ti K$ exists for all $\vphi\in\R$. Thus, $\ti K$ and $K$ are well defined. We now show that $K$ has all the required properties and start with the limit in \eqref{eq:prop_K_s}. From the boundedness of $\Pf-\Ppf$ it follows that the family $((\mu^*_\vphi\Ppf)-\Ppfr \dd r/r)_{\vphi\in\R}$ of 1-forms on $(0,r_0]$ is uniformly bounded. In exact analogy to \cite[Lemma 2]{fu18p}, one may use this uniform boundedness to show that there is a $B\in\R$ such that
$$\forall \vphi\in\R~\forall \rho \in (0, T]:~~~~\left|\left(\frac{\rho}{T}\right)^\zeta\gG_{T}^{~\rho}(\mu_{\vphi}^*\Pf)\right|<2B.$$
Thus, the integrand in $\ti K$ is uniformly bounded and we can conclude that $K$ converges to $\frac{V_+}{\ipl V_+,V_-\ipr}$ as one approaches $\sip$.

Next, we prove \eqref{eq:limit_prim_mink_surf1}. So let $\comp\subset\pb\gD\bsb$ be such that $\comp\cub$ is compact with limit point $\sip$ and choose $p\in\pb\gD\bsb$. By the same arguments as in the proof of Prop \ref{prop:maintool_firstkind_surf}, it is sufficient to prove the statement for $p$ and $\comp$ such that $p\in\comp$. Since we assumed $z(\gD)$ to be a disc of radius $r_0$, we can connect $p$ to a point $q$ by first moving radially along the $r$-parameter line $\pb\mu_{\phi(p)}$ at angle $\phi(p)$ and then along the $\phi$-parameter circle $\pb\gga_{r(q)}$ of radius $r(q)$. Accordingly
\begin{equation}
\left(\frac{r(q)}{r(p)}\right)^\zeta\gG_p^{~q}(\pb\Pf)=\left(\frac{r(q)}{r(p)}\right)^\zeta\gG_{r(p)}^{~r(q)}(\mu^*_{\phi(p)}\Pf)\gG_{\phi(p)}^{~\phi(q)}(\gga^*_{r(q)}\Pf).
\label{eq:factorization_maintool_sing}
\end{equation}
The 1-form $\mu^*_{\phi(p)}\pf$ is a pole form on $(0,r_0]$ and we have defined the map $K$ precisely such that by \cite[Prop 2]{fu18p} the first factor in \eqref{eq:factorization_maintool_sing} has the limit
\begin{equation}
\lqb \left(\frac{r(q)}{r(p)}\right)^\zeta\gG_{r(p)}^{~r(q)}(\mu^*_{\phi(p)}\Pf)=K(p) V_-^*,
\label{eq:limit_poleform_surf_radial_mink}
\end{equation}
where $V_\pm$ are eigenvectors of $\Ppfr$ to the eigenvalues $\pm\zeta$. Now write \eqref{eq:factorization_maintool_sing} in the form
\begin{align*}
\left(\frac{r(q)}{r(p)}\right)^\zeta\gG_p^{~q}(\pb\Pf)=& \left[\left(\frac{r(q)}{r(p)}\right)^\zeta\gG_{r(p)}^{~r(q)}(\mu^*_{\phi(p)}\Pf)-K(p)V_-^*\right]\gG_{\phi(p)}^{~\phi(q)}(\gga^*_{r(q)}\Pf) +\\
&+K(p)\left[\left(\gG_{\phi(q)}^{~\phi(p)}(\gga^*_{r(q)}\Pf)-\gG_{\phi(q)}^{~\phi(p)}(\gga^*_{r(q)}\Ppf)\right)V_-\right]^*+\\
&+K(p)\left[\gG_{\phi(q)}^{~\phi(p)}(\gga^*_{r(q)}\Ppf)V_-\right]^*.
\end{align*}
Due to \eqref{eq:limit_poleform_surf_radial_mink} and the boundedness \eqref{eq:boundedness_of_angpart_purepoleform_curve} of $\gG_{\phi(p)}^{~\phi(q)}(\gga^*_{r(q)}\Pf)$, the first summand on the right hand side converges to zero as $q\in\comp$ approaches $\sip$. Due to \eqref{eq:uniform_cont_angpart_poleform}, also the second summand converges to zero. Since $(\gga^*_T\Ppf)( V_-)=0$ for all $T\in(0,r_0]$, we have $\gG_{\phi(p)}^{~\phi(q)}(\gga^*_{r(q)}\Ppf)V_-=V_-$ and so the third summand equals $K(p)V_-^*$. Therefore, indeed
\begin{equation}
\lqbc \left(\frac{r(q)}{r(p)}\right)^\zeta\gG_p^{~q}(\pb\Pf)= K(p)V_-^*,
\label{eq:maintool_sing_tempo2}
\end{equation}
which proves \eqref{eq:limit_prim_mink_surf1}.  

To show \eqref{eq:limit_prim_mink_surf2}, we note that the map $~^*$ which sends an endomorphism of $\Rmn$ to its adjoint with respect to the Minkowski inner product is continuous. Since $\gG_p^{~q}(\Pf)^*=\gG_q^{~p}(\Pf)$, the limit \eqref{eq:limit_prim_mink_surf2} follows from \eqref{eq:limit_prim_mink_surf1}.

Finally, the first relation in \eqref{eq:prop_K_s} is a consequence of
$$K(p)V_-^*=\lqb \gG_p^{~q}(\Pf)=\gG_p^{~\ti p}(\Pf)\lqb \gG_{\ti p}^{~q}(\Pf)=\gG_p^{~\ti p}(\Pf)K(\ti p)V_-^*.$$
Smoothness of $K$ now follows from \eqref{eq:prop_K_s} and the smoothness of $\pb\pf$.
\end{proof}

The expression \eqref{eq:monodromy_first_kind} for the monodromy of a pole form of the first kind is very useful. It not only shows that its Jordan normal form is the same as that of $\mon(\ppf)$, which can be computed explicitly. When $\pf$ is the 1-form associated to a meromorphically isothermic surface, \eqref{eq:monodromy_first_kind} also gives a geometric meaning to the (generalized) eigenvectors of $\mon_P(\gl\of)$, as we will see in Cor \ref{cor:monodromy_fop}, Cor \ref{cor:monodromy_sop_mink} and Cor \ref{cor:monodromy_sop}. The results that we have found for the monodromy of pole forms of the second kind are less explicit and incomplete. 
\begin{prop}\label{prop:strucutre_monodromy_pfsk}
Let $\pf$ be a pole form on $\gD\bsb$ with $\pf-\ppf$ bounded for a pure pole form $\ppf$ of the second kind. Then, for all $P\in\gD\bsb$ the monodromies $\mon_P(\Pf)$ and $\mon(\Ppf)$ have the same eigenvalues. Moreover, with $\la K\ra$ as in Prop \ref{prop:maintool_sing_s} and $p\in\pb\gD\bsb$ such that $\proj(p)=P$, the line $\la K(p)\ra$ is an eigendirection of $\mon_P(\Pf)$ with eigenvalue $1$. 
\end{prop}
\begin{proof}
Since $\mon_P(\Pf)$ and $\mon_{\ti P}(\Pf)$ are similar for all $P,\ti P\in\gD\bsb$, it is sufficient to prove that $\mon_P(\pf)$ and $\mon(\Ppf)$ have the same eigenvalues for one point $P\in\gD\bsb$. Let $z$ be a coordinate adapted to $\ppf$ and choose $P\in\gD\bsb$ such that $P$ can be joined to $s$ along an $r$-parameter line $\mu$ entirely contained in $\gD$. For $T\in [|z(P)|,0)$, define 
$$\mon(T):=\mon_{\mu(T)}(\pf)~~~\text{and}~~~\mon(0):=\mon(\ppf),$$
such that $\mon(|z(P)|)=\mon_P(\Pf)$. Then $\mon:[|z(P)|,0]\rightarrow \Pro O(\Rmn)$ is continuous. Moreover, $\mon(T_1)$ and $\mon(T_2)$ are similar for all $T_1,T_2\in[|z(P)|,0)$. Thus, the eigenvalues of $\mon(T)$ are the same for all $T\in[|z(P)|,0)$. But since $\mon$ is continuous and eigenvalues of a matrix depend continuously on its entries, we can conclude that the eigenvalues of $\mon(T)$ are the same for all $T\in [T,0]$. In particular, the eigenvalues of $\mon_P(\Pf)$ coincide with those of $\mon(\Ppf)$. 

To see that $\la K(p)\ra$ is an eigendirection of $\mon_P(\Pf)$, write $\Ppf=\Rea(-(\Ppfr+i\Ppfi) \dd z/z)$ and let $V_\pm$ be eigenvectors of $\Ppfr$ to the eigenvalues $\pm\zeta$ with $\zeta>0$. Denote by $K$ the lift of $\la K\ra$ from Prop \ref{prop:maintool_sing_s}. Then, by \eqref{eq:prop_K_s}, we get
$$\forall q\in\gD\bsb:~~~~\mon_P(\Pf)K(p)=\gG_p^{~q}(\Pf)\mon_{\proj(q)}(\Pf)\gG_q^{~p}(\Pf)K(p)=\left|\frac{\pb z(q)}{\pb z(p)}\right|^\zeta\gG_p^{~q}(\Pf)\mon_{\proj(q)}(\Pf)K(q).$$
Now choose $\comp\subset \pb\gD\bsb$ such that $\comp\cub$ is compact with limit point $\sip$ and take the limit $\comp\ni q\rightarrow \sip$ to find
$$\mon_P(\Pf)K(p)=\lqbc \left|\frac{\pb z(q)}{\pb z(p)}\right|^\zeta\gG_p^{~q}(\Pf)\mon_{\proj(q)}(\Pf)K(q)=K(p)V_-^*\,\mon(\Ppf)\frac{V_+}{\ipl V_+,V_-\ipr}=K(p),$$
where we used \eqref{eq:limit_prim_mink_surf1}, \eqref{eq:uniform_cont_angpart_poleform} and \eqref{eq:prop_K_s}.
\end{proof}
In particular, this shows that the pushforward of $\la K\ra$ from $\pb\gD$ to $\gD$ is well defined. However, we cannot conclude from Prop \ref{prop:strucutre_monodromy_pfsk} that $\mon_P(\Pf)$ is diagonalizable (over $\Com$). Namely, although $\mon(\Ppf)$ is a rotation and thus diagonalizable over $\Com$, it could be that $\mon_P(\Pf)$ is similar to a product $e^{V_0\wedge \ti W}\mon(\Ppf)$ with $V_0\in\Li^{n+1}$, spacelike $\ti W$ and $\mon(\Ppf)V_0=V_0$, $\mon(\Ppf)\ti W=\ti W$. In other words, $\mon(\Pf)$ could be the product of a rotation and a commuting translation. Thus, either there is an $(n-2)$-dimensional sphere of points invariant under $\mon_P(\Pf)$ or $\la K(p)\ra$ is the only invariant point.

%-------------------------------------------------------------------
%-------------------------------------------------------------------

% ---------------------------------------------------------------------------------------------
% ---------------------- F I R S T  O R D E R  P O L E ----------------------------------------
% ---------------------------------------------------------------------------------------------

\section{Pole of first order}\label{section:surf_pole_fo}
We now consider the case of a meromorphically isothermic surface $(\la f\ra,\Qua)$ with domain $\gD$ diffeomorphic to a closed disc and such that $\Qua$ has a pole of first order at an interior point $\sip \in\gD$, but is nowhere zero and holomorphic otherwise. For simplicity, we say that $(\la f\ra,\Qua)$ has a pole of first order at $s$. Using \eqref{eq:component_endomorphism}, we may write
$$\of=\of^{(1,0)}+\of^{(0,1)}=2\Rea\left(\of^{(1,0)}\right)=\Rea\left(f\wedge \bar \p f\frac{\Qua}{\ipl \bar\p f,\p f\ipr}\right).$$
Now define
\begin{equation}
\ppf=\Rea\left(f(\sip)\wedge \bar\p_s f(\nu)\frac{\Qua}{\ipl \bar\p_s f(\nu),\p f\ipr}\right)~~~~\text{where}~~~\nu\in T_s\gD\bs\{0\}\subset T_s\gD_\Com.
\label{eq:firstorder_ppf}
\end{equation}
Then $\gl\ppf$ is a degenerate pure pole form for all $\gl\in\R^\times$. Since $\la f\ra$ is smooth and $\Qua$ has a pole of first order at $\sip$, the difference $\of-\ppf$ is bounded and hence $\gl\of$ is a pole form for all $\gl\in\R^\times$.
Now choose $\nu\in T_s\gD$ to be such that 
\begin{equation}
\Res_{\sip}\left(\frac{\Qua}{\ipl \bar\p_s f(\nu),\p f\ipr} \right)=\frac{2i}{\|\dd_s f(\nu)\|^2}.
\label{eq:specify:nu}
\end{equation}
This choice of $\nu$ is independent of the choice of lift $f$ of $\la f\ra$. Then, from Def \ref{defI:pure_pole_forms_s} we find
$$\ppfi=-\Ima\left(f(\sip)\wedge \bar\p_s f(\nu)\Res_{\sip}\left(\frac{\Qua}{\ipl \bar\p_s f(\nu),\p f\ipr} \right)\right)=-f(\sip)\wedge \frac{2 \Rea(\bar \p_s f(\nu))}{\|\dd_s f(\nu)\|^2}=-f(\sip)\wedge \frac{\dd_s f(\nu)}{\|\dd_s f(\nu)\|^2}.$$
Using \eqref{eq:monodromy_purepoleform}, the monodromy of $\gl\ppf$ thus reads
\begin{equation}
\mon(\gl\ppf)=\La\id+2\pi\gl f(\sip)\wedge \frac{\dd_s f(\nu)}{\|\dd_s f(\nu)\|^2}-\frac{(2\pi)^2\gl^2}2\,\frac{f(\sip)f(\sip)^*}{\|\dd_s f(\nu)\|^2} \Ra.
\label{eq:lim_mon_fo_ind}
\end{equation}
We remark that there are precisely two curvature lines that end in the umbilic $\sip$. The tangent directions of these two lines at $\sip$ coincide and are orthogonal to $\la\nu\ra$.

% -------------------------- M O V I N G  I N T O   T H E   U M B I L I C -----------------------
% -----------------------------------------------------------------------------------------------

\subsection{Limiting behaviour of transforms}\label{section:surf_cont_limit_fo}
We can now apply Prop \ref{prop:maintool_firstkind_surf} to the pole forms $\gl\of$, where $\gl\of-\gl\ppf$ is bounded for the degenerate pure pole form $\gl\ppf$, given by \eqref{eq:firstorder_ppf}. In particular, the primitives of the gauge transformed 1-forms $\Gtr{\gl\pb\Ppf}p{\gl\pb\Of}$ defined in \eqref{eq:defi_gaugetrafo_firstkind} have limits at $s$. First we determine the limiting behaviour of the $\gl$-Calapso transforms of $(\la\pb f\ra,\pb\Qua)$.
\begin{theorem}\label{theorem:cal_trafo_surface_fo}
Let $(\la f\ra,\Qua)$ have a pole of first order at $s$ and $(\la f_{\gl,p}\ra,\pb\Qua)$ be the $\gl$-Calapso transform normalized at $p\in\pb \gD\bsb$ of $(\la \pb f\ra,\pb\Qua)$. For a smooth lift $f$ of $\la f\ra$, let the lift $f_{\gl,p}$ be given by
\begin{equation}
f_{\gl,p}=\gG_p(\gl\pb\Of)\pb f
\label{eq:bounded_nonzerolimitlift_caltrafo_s_fo}
\end{equation}
and choose $\comp\subset\pb\gD\bsb$ such that $\comp\cub\subset\pb\gD$ is compact with limit point $\sip$. 
\begin{enumerate}
\item The lift $f_{\gl,p}$, restricted to $\comp$, has the limit $\gG_p^{~\sip}(\Gtr{\gl\pb\Ppf}p{\gl\pb\Of})f(\sip)$ at $\sip$. In particular, $\la f_{\gl,p}\ra$, restricted to $\comp$, has a limit at $\sip$. 

\item Denote by $\tandir$ a smooth, nowhere zero tangent vector field on $\gD$ and by $\pb\tandir$ its pullback to $\pb\gD\bsb$. Then $\|\dd f_{\gl,p}(\pb\tandir)\|^2$ has a nonzero limit at $\sip$.

\item With $\pb\tandir$ as above,
\begin{equation}
\lqbc f_{\gl,p}(q)\wedge \dd_q f_{\gl,p}(\pb\tandir)=\gG_p^{~\sip}(\Gtr{\gl\pb\Ppf}p{\gl\pb\Of})~\big(f(\sip)\wedge \dd_s f(\tandir)\big)~\gG_s^{~p}(\Gtr{\gl\pb\Ppf}p{\gl\pb\Of}).
\label{eq:convergence_caltrafo_tangent_wedge}
\end{equation}
In particular, the limit of the tangent congruence of $\la f_{\gl,p}\ra$ is
$$\lqbc\La f_{\gl,p}(q),\dd_q f_{\gl,p}\Ra=\gG_p^{~\sip}(\Gtr{\gl\pb\ppf}p{\gl\pb\of})\La f(\sip),\dd_s f\Ra.$$

\end{enumerate}

\end{theorem}
\begin{proof}

\begin{enumerate}
\item 
%The $\gl$-Calapso transform normalized at $p\in\pb\gD\bsb$ is given by $\la f_{\gl,p}\ra=\gG_p(\gl\pb\of)\la \pb f\ra$. Looking at \eqref{eq:expo_xi_at_b_fo}, we see that the smooth lift $f$ of $\la f\ra$ satisfies all the conditions of Cor \ref{cor:deg_aid_s}. Thus, $\la f_{\gl,p}\ra$ has a limit at $\sip$. 
Let $z$ be a coordinate adapted to $\ppf$. Then, by Lemma \ref{lemma:primitives_purepoleforms_s} and the smoothness of $f$, there is a constant $B\in\R$ such that
$$\forall q\in\comp:~~~~~|\gG_p^{~q}(\gl\pb\Ppf)(\pb f(q)-f(\sip))|\leq B\,\bd(p,q)\,|\pb z(q)|.$$
Since $\lqbc \bd(p,q)|\pb z(q)|=0$ and $\Ppf(f(\sip))=0$, we find
\begin{equation}
\lqbc \gG_p^{~q}(\gl\pb\Ppf)\pb f(q)=\lqbc \gG_p^{~q}(\gl\pb\Ppf) f(\sip)=f(\sip).
\label{eq:limit_primofxi_f}
\end{equation}
Therefore, using the factorization \eqref{eq:factorization_into_surf}, we find
$$\lqbc f_{\gl,p}(q)=\lqbc \gG_p^{~q}(\Gtr {\gl\pb\Ppf}p {\gl\pb\Of})\gG_p^{~q}(\gl\pb\Ppf)\pb f(q)=\gG_p^{~\sip}(\Gtr {\gl\pb\Ppf}p {\gl\pb\Of})f(\sip).$$

\item With $f_{\gl,p}$ given by  \eqref{eq:bounded_nonzerolimitlift_caltrafo_s_fo}, since $\Of(f)=0$ and $\gG_p^{~q}(\gl\pb\Of)\in O(\Rmn)$, we get
$$\|\dd f_{\gl,p}(\pb\tandir)\|^2=\|\dd\pb f(\pb\tandir)\|^2=\|\dd f(\tandir)\|^2,$$
which has a finite limit at $\sip$. 

\item To see \eqref{eq:convergence_caltrafo_tangent_wedge}, we write
\begin{equation}
f_{\gl,p}\wedge \dd f_{\gl,p}(\pb\tandir)=\gG_p(\Gtr{\gl\pb\Ppf}p{\gl\pb\Of})\gG_p(\gl\pb\Ppf)~\big(\pb f\wedge \dd \pb f(\pb\tandir)\big)~\gG^{~p}(\gl\pb\Ppf)\gG^{~p}(\Gtr{\gl\pb\Ppf}p{\gl\pb\Of}).
\label{eq:wedge_caltrafo_fo_s1}
\end{equation}
Similarly as in \eqref{eq:limit_primofxi_f} we find
\begin{align}
\lqbc \gG_p^{~q}(\gl\pb\Ppf)~\big(\pb f(q)\wedge \dd_q \pb f(\pb\tandir)\big)~\gG_q^{~p}(\gl\pb\Ppf)&=\lqbc \gG_p^{~q}(\gl\pb\Ppf)~\big(f(\sip)\wedge \dd_s  f(\tandir)\big)~\gG_q^{~p}(\gl\pb\Ppf)\notag\\
&=f(\sip)\wedge \dd_s  f(\tandir)\label{eq:conv_caltrafo_tan},
\end{align}
where in the first equality we used smoothness of $f$ and $\tandir$, and in the second equality we used $\Ppf(\dd_s f(\gs))\in \la f(\sip)\ra$. Substituting \eqref{eq:conv_caltrafo_tan} into the limit $\comp\ni q\rightarrow \sip$ of \eqref{eq:wedge_caltrafo_fo_s1} yields the result. 

The statement about the tangent congruence of $\la f_{\gl,p}\ra$ follows readily from \eqref{eq:convergence_caltrafo_tangent_wedge}.

\end{enumerate}

\end{proof}

We now turn to the limiting behaviour of Darboux transforms of $(\la \pb f\ra,\pb\Qua)$ at the umbilic $\sip$.
\begin{theorem}\label{theorem:darb_trafo_surface_fo}
Let $(\la f\ra,\Qua)$ have a pole of first order at $s$.  Choose $\comp\subset \pb \gD\bsb$ such that $\comp\cub\subset\pb\gD$ is compact with limit point $\sip$. Then, all Darboux transforms of $(\la \pb f\ra,\pb\Qua)$, restricted to $\comp$, have the limit $\la f(\sip)\ra$ at $\sip$.
\end{theorem}
%----------------------------  P R O O F
\begin{proof}
Let $(\la \fh\ra,\pb\Qua)$ be a $\gl$-Darboux transform of $(\la \pb f\ra,\pb\Qua)$ and choose $p\in\pb\gD\bsb$. If $\la \fh(p)\ra$ is different from $\gG_p^{~\sip}(\Gtr{\gl\pb\ppf}p{\gl\pb\of})\la  f(\sip)\ra$, the limit point of the $\gl$-Calapso transform normalized at $p$, then, using the factorization \eqref{eq:factorization_into_surf} and \eqref{eq:lim_ex_deg_s} we find
$$\lqbc \la \fh(q)\ra=\lqbc \gG_q^{~p}(\gl\pb\ppf)\gG_q^{~p}(\Gtr{\gl\pb\ppf}p{\gl\pb\of})\la \fh(p)\ra=\la f(\sip)\left( \gG_p^{~\sip}(\Gtr{\gl\pb\Ppf}p{\gl\pb\Of})f(\sip)\right)^*\ra\la \fh(p)\ra=\la f(\sip)\ra.$$
If on the other hand $\la \fh(p)\ra=\gG_p^{~\sip}(\Gtr{\gl\pb\ppf}p{\gl\pb\of})\la f(\sip)\ra=\lim_{\comp\ni\ti q\rightarrow \sip}\gG_p^{~\ti q}(\gl\pb\of)\la f(\ti q)\ra$, we use Thm \ref{theorem:cal_trafo_surface_fo} and \eqref{eq:convergence_primofomegati} to find 
$$\lqbc \la \fh(q)\ra=\lqbc \gG_q^{~p}(\gl\pb\of)\lim_{\comp\ni\ti q\rightarrow \sip}\gG_p^{~\ti q}(\gl\pb\of)\la f(\ti q)\ra=\lqbc \la f_{\gl,q}(s)\ra=\la f(\sip)\ra.$$

\end{proof}

% -----------------------------------------------------------------------------------------------
% -------------------------- M O V I N G  A R O U N D   T H E   U M B I L I C -------------------
% -----------------------------------------------------------------------------------------------

\subsection{Monodromy of transforms}\label{section:surf_monodromy_fo}

An expression for the monodromy of $\gl\of$ is directly obtained from Prop \ref{prop:maintool_firstkind_surf}.

\begin{corollary}\label{cor:monodromy_fop}
Let $(\la f\ra,\Qua)$ have a pole of first order at $s$ and denote by $\of$ its associated 1-form. Then, for $P\in\gD\bsb$ and $p\in\pb\gD\bsb$ such that $\proj(p)=P$, the monodromy of $\gl\of$ with base point $P$ is
\begin{equation}
\mon_{P}(\gl\of)=\La \id+2\pi\gl  \lqbc \frac{f_{\gl,p}(q)\wedge  \dd_q f_{\gl,p}(\ti\nu)}{\|\dd_q f_{\gl,p}(\ti\nu)\|^2}-2\pi^2\gl^2\lqbc\frac{f_{\gl,p}(q)f_{\gl,p}(q)^*}{\|\dd_q f_{\gl,p}(\ti\nu)\|^2}\Ra,
\label{eq:monodromy_first_order_pole_prop}
\end{equation}
where $f_{\gl,p}$ is any lift of $\la f_{\gl,p}\ra$, the subset $\comp\subset\pb\gD\bsb$ is such that $\comp\cub\subset\pb\gD$ is compact with limit point $\sip$ and $\ti\nu$ is any smooth section of $T\gD$ such that $\ti\nu(s)=\nu$ with $\nu$ as in \eqref{eq:specify:nu}. In particular, $\la f_{\gl,p}(\sip)\ra$ is the only null eigenspace of any power $(\mon_P(\gl\of))^j$ with $j\in\mathbb N$.
\end{corollary}
\begin{proof}
This simply follows from \eqref{eq:lim_mon_fo_ind}, Thm \ref{theorem:cal_trafo_surface_fo} and the fact that the only null eigendirection of a Lie algebra element $v\wedge w$ with degenerate $\la v,w\ra$ is $\la v,w\ra\cap\Li^{n+1}$. 
\end{proof}

\begin{corollary}\label{cor:monodromy_firstorder}
Let $(\la f\ra,\Qua)$ have a pole of first order at $s$. Then no $\gl$-Calapso transform of $(\la\pb f\ra,\pb\Qua)$ with $\gl\neq 0$ can be pushed forward to any $j$-fold cover of $\gD\bsb$ with $j\in\mathbb N$.
\end{corollary}
\begin{proof}
This follows directly from Prop \ref{prop:monodromy_dar_cal_trafo} and Cor \ref{cor:monodromy_fop}: In order for the pushforward to the $j$-fold cover of $\gD\bsb$ to be defined, the image of $\la f_{\gl,p}\ra$ would need to be invariant under $(\mon_p(\gl\of))^j$. But the only point invariant under $(\mon_p(\gl\of))^j$ is $\la f_{\gl,p}(\sip)\ra$ and $\la f_{\gl,p}\ra$ is certainly not constant.
\end{proof}
In the same way, we get
\begin{corollary}
Let $(\la f\ra,\Qua)$ have a pole of first order at $s$. The pushforward of a $\gl$-Darboux transform $(\la \fh\ra,\pb\Qua)$ of $(\la \pb f\ra,\pb\Qua)$ to the $j$-fold cover of $\gD\bsb$ is well defined if and only if $\la \fh(p)\ra=\la f_{\gl,p}(\sip)\ra$ for some $p\in\pb\gD\bsb$.
\end{corollary}
%Fig \ref{fig:darelli} shows two Darboux transforms of an ellipsoid.
\begin{figure}[h]
\centering
~~~~~~\begin{subfigure}[t]{0.45\textwidth}
\centering
\includegraphics[height=190pt]{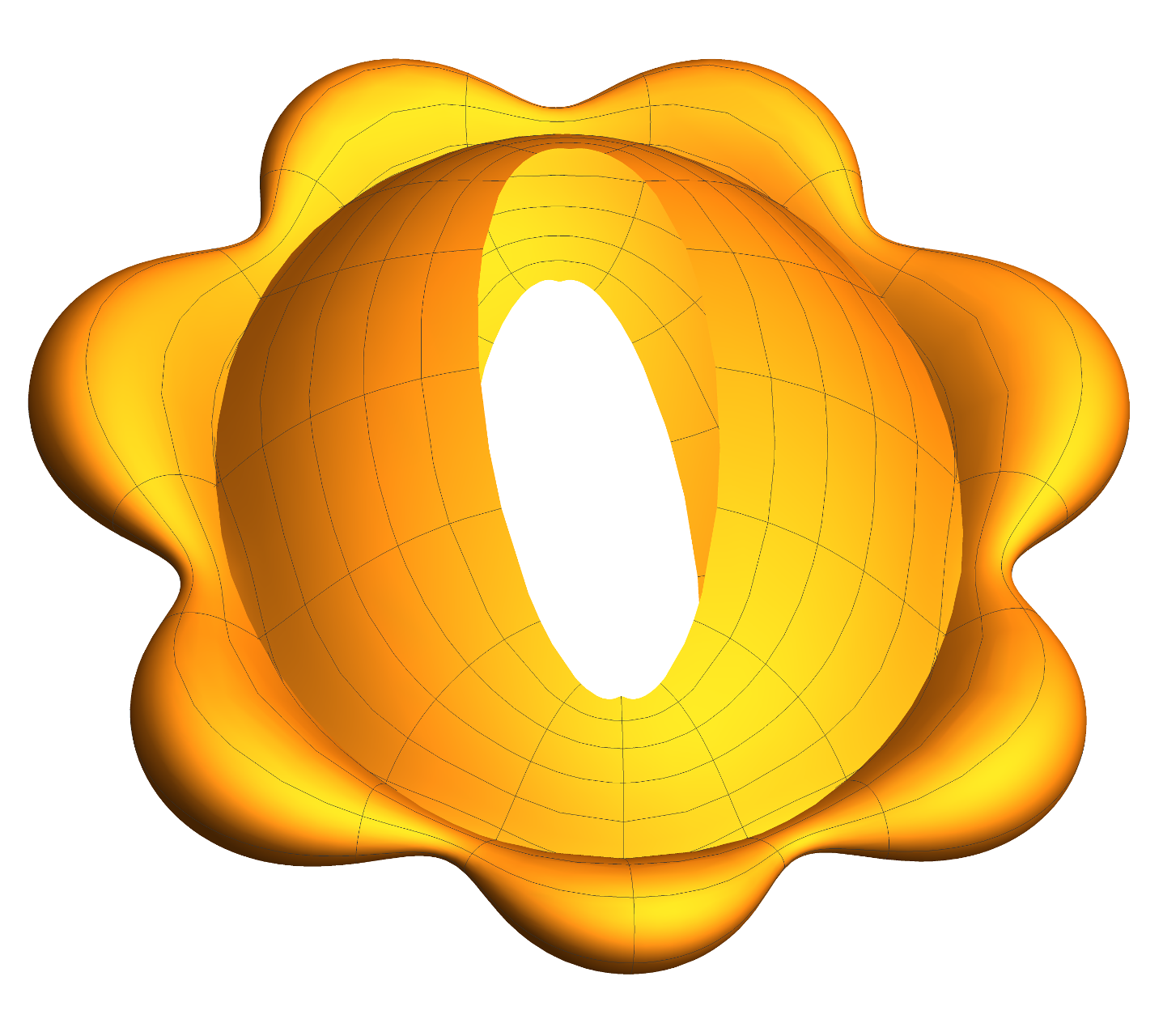}
\end{subfigure}~~~~~~~~
\begin{subfigure}[t]{0.45\textwidth}
\centering
\includegraphics[height=190pt]{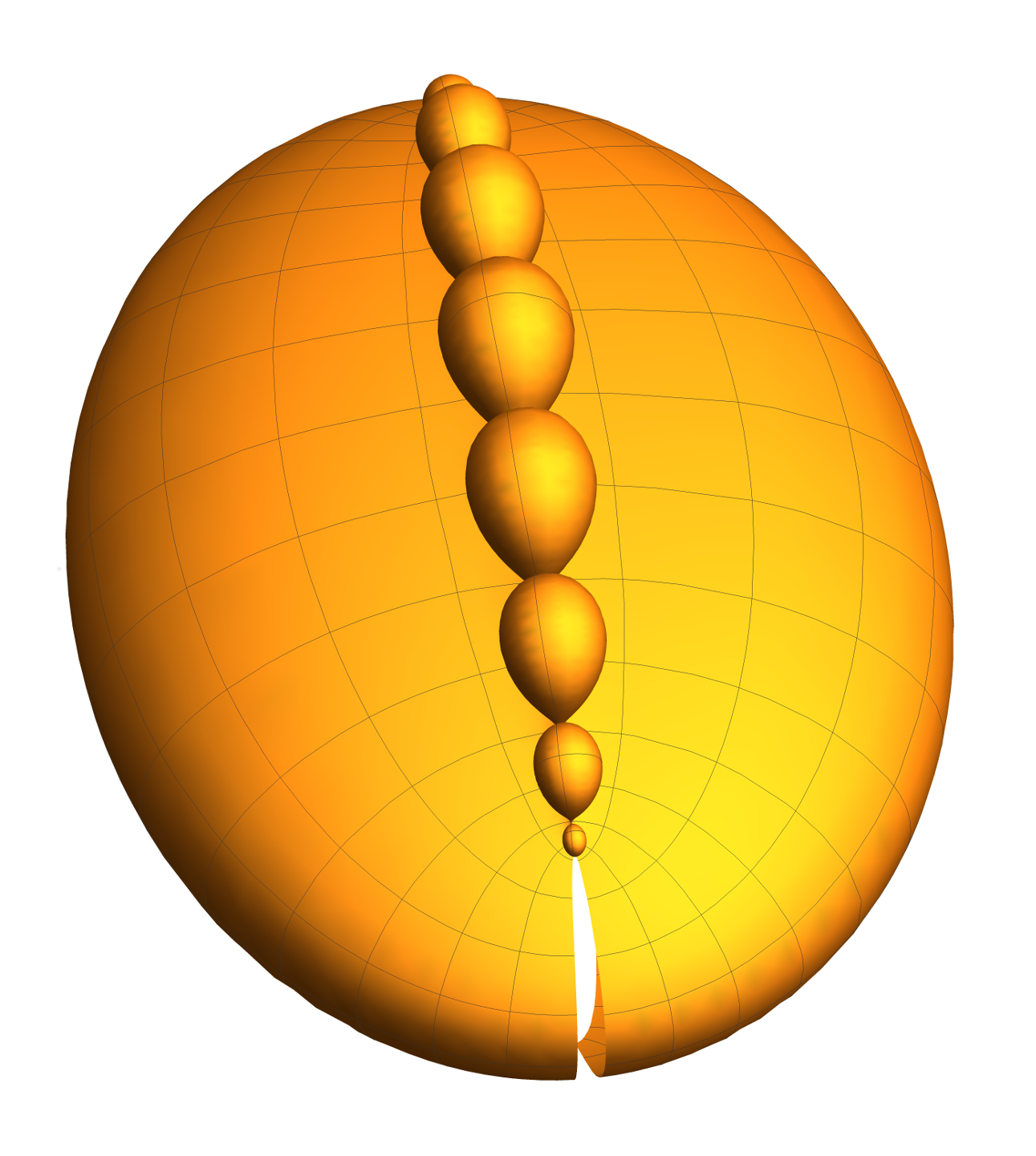}
\end{subfigure}
\caption{Darboux transforms of an ellipsoid.}\label{fig:darelli}
\end{figure}

%--------------------------------------------------------------------------------------------------
%---------------------- P O L E  O F  S E C O N D   O R D E R -------------------------------------	
%--------------------------------------------------------------------------------------------------
\section{Pole of second order}\label{section:surf_pole_so}
We now come to the case of a meromorphically isothermic surface $(\la f\ra,\Qua)$ with simply connected domain $\gD$ and such that $\Qua$ has a pole of second order at an interior point $\sip \in\gD$, but is nowhere zero and holomorphic otherwise. For simplicity, we say that $(\la f\ra,\Qua)$ has a pole of second order at $s$. In this case, the 1-form $\of$ associated to $(\la f\ra,\Qua)$ has a pole of second order at $\sip$. But we can gauge transform every $\gl\of$ using a singular map $\hyfr$ to a 1-form $\gatr{\hyfr}{\gl\of}$ that has only a pole of first order.

As explained below Defs \ref{defi:iso_surf} and \ref{defi:trafos_surf}, we may replace $\Qua$ by any nonzero, real multiple of it and thereby achieve that the holomorphic 1-form $\sqrt{\Qua}$ has residue $1$ at $s$. Denote by $z$ a coordinate adapted to $\sqrt{\Qua}$, such that $\Qua=\dd z^2/z^2$. As before, we assume $z$ to be defined on all of $\gD$. Polar coordinates $(r,\phi)$ associated to $z$ are curvature line coordinates for $\la \pb f\ra=\proj^*\la f\ra$ because
$$\pb\Qua=\frac{\left(\dd r \,e^{i\phi}+r\,i\dd\phi \,e^{i\phi}\right)^2}{r^2 e^{2i\phi}}=\frac{\dd r^2}{r^2}-\dd\phi^2+\frac{2i}{r}\dd r\dd\phi,$$
such that $\pb\Qua$ and thus also the Hopf differential $\pb\Qho$ restricted to $\phi=const.$ or $r=const.$ are real (cf. \cite[Ch VI, Sect 1.2]{hop89} or \cite[Lemma 1.1.5]{fu18}). Conformal curvature line coordinates on $\pb\gD\backslash\{\sip\}$ are thus given by $(\rho, \phi)$ with $\rho=\ln(r)$, that is, $\pb z=e^{\rho+i\phi}$.

The map $\hyfr$ that we use for the gauge transformation has the form of a product $\hyfr= \la F\ra \la R\ra$. The first factor $\la F\ra:\gD\bsb\rightarrow \Pro O(\Rmn)$ is constructed with a smooth frame $F$ for the flat lift $f$ of $\la f\ra$ with respect to $z$. We choose $F$ such that
\begin{equation}
\forall p\in\gD\bsb:~~~~F(p)\baszer=f(p),~~~F(p)\frac 12 (\bastanu-i\bastanv)=\frac{\p f}{\p z}(p),~~~F(p)\basinf\in \La \frac{\p f}{\p u}(p),\frac{\p f}{\p v}(p)\Ra^\perp,
\label{eq:props_of_frame}
\end{equation}
where $\baszer,\basinf\in\Li^{n+1}$ with $\ipl \baszer,\basinf\ipr=-1$ and $\bastanu$, $\bastanv$ are two orthonormal vectors in $\la \baszer,\basinf\ra^\perp$. We define the second factor $\la R\ra:\gD\bsb\rightarrow \Pro O(\Rmn)$ by
$$R\,(\bastanu-i\bastanv)=\frac{z}{|z|}(\bastanu-i\bastanv),~~~~R\,\baszer=\frac 1 {|z|}\,\baszer,~~~~R\,\basinf=|z|\,\basinf,~~~~R\big|_{\la \baszer,\basinf,\bastanu,\bastanv\ra^\perp}=\id.$$
It is singular at $s$ and so the product $g=\la F\ra\la R\ra=\la FR\ra$ is a singular frame for $\la f\ra$ with limit
\begin{equation}
\lqb \hyfr(q)=\lqb\la F(q)R(q)\ra=\lqb\la |z(q)|F(q)R(q)\ra=\la f(\sip)\basinf^*\ra.
\label{eq:limit_singtrafo}
\end{equation}
The $(1,0)$-component of the gauge transformed 1-form $\gatr \hyfr {\gl\of}$ then computes to
$$(\gatr \hyfr {\gl\of})^{(1,0)}=R^{-1}F^{-1}\gl\of^{(1,0)} FR+R^{-1}F^{-1}\p F R+R^{-1}\p R$$
with
\begin{align*}
R^{-1}F^{-1}\gl\of^{(1,0)} FR=&R^{-1}\frac{\gl}2\baszer\wedge\left(\bastanu+i\bastanv\right)\frac{\dd z}{z^2}\,R+\R\,\id\\
=&\frac{\gl}2 \,\baszer\wedge\left(\bastanu+i\bastanv\right)\frac{\dd z}{z}+\R\,\id,\\
R^{-1}F^{-1}\p F R=&R^{-1}\left(\Fr^{\perp\perp}-\basinf\wedge (\bastanu-i\bastanv)\frac{\dd z}2+\baszer\wedge \Fr^\perp\right)R\\
=&R^{-1}\Fr^{\perp\perp} R-\basinf\wedge (\bastanu-i\bastanv)\frac {\dd z}{2z}+|z|\,\baszer\wedge R^{-1}\Fr^\perp,\\
R^{-1}\p R=&\frac{\dd z}{2z}\,\left(\basinf\wedge\baszer-i\,\bastanu\wedge \bastanv\right),
\end{align*}
where $\Fr^{\perp\perp}$ is the projection of $F^{-1}\p F$ onto $\gL^2\la \baszer,\basinf\ra_\Com^\perp$ and the 1-form $\Fr^\perp$ is $\la \baszer,\basinf\ra_\Com^\perp$-valued. In particular, $R^{-1} \Fr^{\perp\perp} R$ and $R^{-1}\Fr^\perp$ are bounded 1-forms on $\gD\bsb$ with respect to $|\cdot|$. Thus,
\begin{equation}
\gatr \hyfr {\gl\of}=\ppf_\gl+2\Rea\left(R^{-1}\Fr^{\perp\perp}R+|z|\,\baszer\wedge R^{-1}\Fr^\perp\right)+\R\,\id
\label{eq:form_of_tieta_so_s}
\end{equation}
with the pure pole form
\begin{equation}\label{eq:form_of_xila_so_s}
\begin{split}
\ppf_\gl:=&\Re\left(\Big(\left(\baszer-\bastanu\right)\wedge\left(\basinf-\gl\bastanu\right)+i \left(\bastanu-\gl\baszer-\basinf\right)\wedge\bastanv\Big)\left(-\frac{\dd z}{z}\right)\right)+\R\,\id\\
=:&\Rea\left((\ppfr+i\ppfi)\left(-\frac{\dd z}{z}\right)\right).
\end{split}
\end{equation}
Clearly, $\gatr \hyfr {\gl\of}-\ppf_\gl$ is bounded, such that $\gatr\hyfr{\gl\of}$ is a pole form. 

For later use, we record the eigenvalues and -vectors of $\ppfr$ and $\ppfi$. Direct computation shows that
\begin{align}
\Ppfr(V_\pm)&=\pm\sqrt{1-2\gl}\,V_\pm~~~\text{with}&~V_\pm&=\gl\baszer+\basinf-2\gl\bastanu\pm\sqrt{1-2\gl}\,(\gl\baszer-\basinf),\\
\Ppfi(W_\pm)&=\pm\sqrt{2\gl-1}\,W_\pm~~~\text{with}&~W_\pm&=\gl\baszer+\basinf-\bastanu\pm\sqrt{2\gl-1}\,\bastanv.
\end{align}
Thus, $\ppf_\gl$ is Minkowski, degenerate or spacelike for $1-2\gl$ greater than, equal to or smaller than zero, respectively. Using \eqref{eq:factori_xi}, the primitives of $\pb\ppf_\gl$ are explicitly given by
\begin{align}
\gG_p(\pb\ppf_\gl)=&\Bigg\langle \left(\frac{r}{r(p)}\right)^{\sqrt{1-2\gl}}\Pi_{\la V_+\ra}+\left(\frac{r}{r(p)}\right)^{-\sqrt{1-2\gl}}\Pi_{\la V_-\ra}+\label{eq:expo_xila_so_s}\\
&+e^{\sqrt{2\gl-1}(\phi-\phi(p))}\Pi_{\la W_+\ra}+e^{-\sqrt{2\gl-1}(\phi-\phi(p))}\Pi_{\la W_-\ra}+\Pi_{\la V_+,V_-,W_+,W_-\ra^\perp}\Bigg\rangle\notag,
\end{align}
where $(r,\phi)$ are polar coordinates associated to $z$. 

The conditions \eqref{eq:props_of_frame} on the frame $F$ do not determine it uniquely. The following Lemma shows that we can choose it such that not only $\gatr \hyfr {\gl\of}-\ppf_\gl$, but even $(\gatr \hyfr {\gl\of}-\ppf_\gl)/|z|$ is bounded.
\begin{lemma}\label{lemma:nice_frame_s}
Let $\la \curdir\ra$ be a smooth curvature direction field of $\la f\ra$ on $\gD\bsb$. The frame $F$ can be chosen such that, additionally to \eqref{eq:props_of_frame}, it frames the $\la \curdir\ra$-curvature sphere congruence, that is\footnote{Here, $(\curdir\cdot\dd)^2f:=\dd(\dd f(\curdir))(\curdir)$.},
\begin{equation}
F\la \baszer,\basinf,\bastanu,\bastanv\ra=\la f,\dd f,(\curdir\cdot\dd)^2f\ra,
\label{eq:condition1_niceframe}
\end{equation}
and $\frac 1 {|z|} (\gatr{\hyfr}{\gl\of}-\ppf_\gl)$ is bounded for any holomorphic coordinate $z$ around $\sip$.
\end{lemma}
\begin{proof}
Let $N_1,...,N_{n-2}$ be lifts of parallel orthonormal normal fields such that
$$\la f,\dd f,(\curdir\cdot\dd)^2f\ra=\la N_1,...,N_{n-2}\ra^\perp.$$
Now choose the frame $F$ such that \eqref{eq:props_of_frame} and \eqref{eq:condition1_niceframe} hold and such that it maps a constant orthonormal basis $(\basnor_1,...,\basnor_{n-2})$ of $\la \baszer,\basinf,\bastanu,\bastanv\ra^\perp$ to $(N_1,...,N_{n-2})$. In order to show that $\frac 1 {|z|} (\gatr{\hyfr}{\gl\of}-\ppf_\gl)$ is bounded, by \eqref{eq:form_of_tieta_so_s} we need to show that $|z|^{-1}2\Rea(
\Fr^{\perp\perp})=|z|^{-1}(F^{-1}\dd F)\big|_{\gL^2\la \baszer,\basinf\ra^\perp}$ is bounded. Since we chose $f$ flat with respect to $z=u+iv$ and since the normal fields $N_i+\la f\ra$ are parallel, we find that
$$(F^{-1}\dd F)\big|_{\gL^2\la \bastanu,\bastanv\ra}=0,~~~~~~(F^{-1}\dd F)\big|_{\gL^2\la \basnor_1,...,\basnor_{n-2}\ra}=0,$$
respectively. Using this, one readily finds that boundedness of $|z|^{-1}(F^{-1}\dd F)\big|_{\gL^2\la \baszer,\basinf\ra^\perp}$ is equivalent to boundedness of the rescaled shape operators $|z|^{-1}A^f_{N_i}$ for $i=1,...,n-2$. Now choose lifts $\bar N_i$ of $N_i+\la f\ra$ that are smooth on all of $\gD$ and denote by $\gk_i$ the eigenvalues of $A^f_{\bar N_i}$ corresponding to the common eigendirection $\la \curdir\ra$. We then find $N_i=\bar N_i+\gk_i f$ and hence
\begin{equation}
|z|^{-1}A^f_{N_i}=|z|^{-1}\left(A^f_{\bar N_i}-\gk_i\,\id\right)=|z|^{-1}\left(A^f_{\bar N_i}-A^f_{\bar N_i}(\sip)\right)+|z|^{-1}(\gk_i(\sip)-\gk_i)\,\id,
\label{eq:bounded_resc_shapeoperator}
\end{equation}
where we used that $\sip$ is an umbilic of $\la f\ra$ and thus $A^f_{\bar N_i}(\sip)=\gk_i(\sip)\,\id$. That \eqref{eq:bounded_resc_shapeoperator} is bounded now follows from smoothness of $A^f_{N_i}$ on $\gD$ and Lipschitz continuity of the eigenvalues of a smooth endomorphism.
\end{proof}
We separately consider the cases $1-2\gl>0$ and $1-2\gl<0$  in the next two sections. We claim without proof that the results in the case $1-2\gl=0$ qualitatively resemble the case $1-2\gl>0$.
%------------------------------------------------------------------------------------------
%----------------------------      negative Lambda
%------------------------------------------------------------------------------------------
\subsection{The behaviour at and around the singularity for $1-2\gl>0$}\label{section:sopole_posla}
For $\gl\in\R^\times$ with $1-2\gl>0$, the pure pole form $\ppf_\gl$ given by \eqref{eq:form_of_xila_so_s} is Minkowski.
\subsubsection{Limiting behaviour of transforms}

\begin{theorem}\label{theorem:cal_trafo_surface_so}
Let $(\la f\ra,\Qua)$ have a pole of second order at $s$. For $\gl\in\R^\times$ with $1-2\gl>0$, let $(\la f_{\gl,p}\ra,\pb\Qua)$ be the $\gl$-Calapso transform of $(\la\pb f\ra,\pb\Qua)$ normalized at $p\in \pb\gD\bsb$ and $\comp\subset\pb\gD\bsb$ be such that $\comp\cub\subset\pb\gD$ is compact with limit point $\sip$.

\begin{enumerate}
	\item The limit of $\la f_{\gl,p}\ra$ restricted to $\comp$ at $\sip$ is $\pb\hyfr(p)\la K_\gl(p)\ra$, where $\la K_\gl(p)\ra$ is the image\footnote{See Prop \ref{prop:maintool_sing_s}.} of the map $\lim_{\comp\ni q\rightarrow \sip}\gG_p^{~q}(\gatr{\pb\hyfr}{\gl\pb\of})$.
	\item If $0<1-2\gl<1$, let $\la \pb\curdir\ra$ be the pullback to $\pb\gD\bsb$ of any smooth curvature sphere congruence $\la \curdir\ra$ of the restriction of $\la f\ra$ to $\gD\bsb$. Then the $\la\pb\curdir\ra$-curvature sphere congruence of $\la f_{\gl,p}\ra$ has the limit
	\begin{equation}
	\lqbc \la f_{\gl,p}(q),\dd_q f_{\gl,p},(\pb\curdir\cdot\dd)_q^2f_{\gl,p}\ra= \pb\hyfr(p)\gG_p^{~\sip}(\Gtr{\gl\pb\ppf}{p}{\gatr{\pb\hyfr}{\gl\pb\of}})\la \baszer,\bastanu,\bastanu,\basinf\ra,
	\label{eq:lim_curvature_sphere_caltrafo_so}
	\end{equation}
	where the frame $F$ in $\hyfr=\la F\, R\ra$ is chosen such that \eqref{eq:condition1_niceframe} holds.
\end{enumerate}
\end{theorem}

\begin{proof}
\begin{enumerate}
	\item Using \eqref{eq:gauge_trafo_beha}, Prop \ref{prop:maintool_sing_s} and $\ipl V_-,\baszer\ipr\neq 0$, we find
$$\lqbc \la f_{\gl,p}(q)\ra=\lqbc \gG_p^{~q}(\gl\pb\of)\la \pb f(q)\ra=\lqbc \pb\hyfr(p)\gG_p^{~q}(\gatr{\pb \hyfr}{\gl\pb\of})\la \baszer\ra= \pb\hyfr(p) \la K_\gl(p)\ra.$$
\item Let $F$ satisfy \eqref{eq:condition1_niceframe}. For $0<1-2\gl<1$, the pure pole form $\ppf_\gl$ is of the first kind. Using the factorization \eqref{eq:factorization_into_surf}, we then find
$$\la f_{\gl,p}(q),\dd_q f_{\gl,p},(\pb\curdir\cdot\dd)_q^2f_{\gl,p}\ra=\pb\hyfr(p) \gG_p^{~q}(\Gtr{\gl\pb\ppf}{p}{\gatr{\pb\hyfr}{\gl\pb\of}})\gG_p^{~q}(\pb\ppf_\gl)\la \baszer,\bastanu,\bastanu,\basinf\ra.$$
The result now follows because $\la \baszer,\bastanu,\bastanu,\basinf\ra$ is invariant under $\gG_p^{~q}(\pb\ppf_\gl)$ and $\gG_p^{~q}(\Gtr{\gl\pb\ppf}{p}{\gatr{\pb\hyfr}{\gl\pb\of}})$ converges as $q$ approaches $\sip$ in $\comp$ by Prop \ref{prop:maintool_firstkind_surf}.
\end{enumerate}
\end{proof}
We remark that for $0<1-2\gl<1$ it follows directly from the convergence of $\la f_{\gl,p}\ra$ and any of its smooth curvature sphere congruences that also the tangent congruence of $\la f_{\gl,p}\ra$ has a limit as one approaches $\sip$ inside $\comp$.

\begin{theorem}
Let $(\la f\ra,\Qua)$ have a pole of second order at $s$. Choose $\comp\subset\pb\gD\bsb$ such that $\comp\cub$ is compact with limit point $\sip$. Then any $\gl$-Darboux transform of $(\la\pb f\ra,\pb\Qua)$ with $1-2\gl>0$ converges to $\la f(\sip)\ra$ as one approaches $\sip$ in $\comp$. 
\end{theorem}
\begin{proof}
Choose $p\in \comp$. We distinguish two cases. Let first $\la \fh(p)\ra$ be distinct from $ \pb\hyfr(p) \la K_\gl(p)\ra$. Then, using \eqref{eq:gauge_trafo_beha}, \eqref{eq:limit_singtrafo}, Prop \ref{prop:maintool_sing_s} and $\ipl \basinf,V_-\ipr\neq 0$ yields
\begin{align*}
\lqbc \la \fh(q)\ra&=\lqbc \pb\hyfr(q) \gG_q^{~p}(\gatr{\pb\hyfr}{\gl\pb\of})\ra \pb\hyfr(p)^{-1}\la \fh(p)\ra\\
&=\la f(\sip)\basinf^*\ra \la V_-K_\gl(p)^*\ra \pb\hyfr(p)^{-1}\la \fh(p)\ra=\la f(\sip)\ra.
\end{align*}
If on the other hand $\la \fh(p)\ra$ is equal to $\pb\hyfr(p)\la K_\gl(p)\ra$, then we use \eqref{eq:prop_K_s} and $\ipl \basinf,V_+\ipr\neq 0$ to find
$$\lqbc \la \fh(q)\ra=\lqbc  \pb\hyfr(q) \gG_q^{~p}(\gatr{\pb\hyfr}{\gl\pb\of})\la K_\gl(p)\ra=\lqbc \pb\hyfr(q) \la K_\gl(q)\ra=\la f(\sip)\ra.$$
\end{proof}

\subsubsection{Monodromy of transforms}\label{section:surf_monodromy_so_posla}
For spectral parameter $0<1-2\gl<1$, the pure pole form $\ppf_\gl$, given by \eqref{eq:form_of_xila_so_s}, is of the first kind. An expression for its monodromy is thus obtained directly from Prop \ref{prop:maintool_firstkind_surf}.

\begin{corollary}\label{cor:monodromy_sop_mink}
Let $(\la f\ra,\Qua)$ have a pole of second order at $s$. For $\gl\in\R$ with $0<1-2\gl<1$ and points $p\in\pb\gD\bsb$ and $P\in\gD\bsb$ with $\proj(p)=P$, the monodromy of $\gl\of$ with base point $P$ is
\begin{equation}
\mon_{P}(\gl\of)=\La \Pi_{\la W(\gl,p)),\bar W(\gl,p)\ra^\perp}+e^{-2\pi i\sqrt{1-2\gl}}\Pi_{\la W(\gl,p)\ra}+e^{2\pi i\sqrt{1-2\gl}}\Pi_{\la \bar W(\gl,p)\ra}\Ra,
\label{eq:monodromy_secondorder_mink_firstkind}
\end{equation}
where 
\begin{equation}
\la W(\gl,p)\ra=\pb\hyfr(p)\gG_{p}^{~\sip}(\Gtr{\pb\ppf_\gl}{p}{\gatr {\pb\hyfr} {\gl\pb\of}})\la W(\gl)\ra,~~~~W(\gl)=i\sqrt{1-2\gl}\,\bastanv\pm(\gl\baszer+\basinf-\bastanu)\in\Com^{n+2}.
\label{eq:defi_wpm_s_mink}
\end{equation}
All points on the $(n-2)$-dimensional sphere represented by $\la \Rea(W(\gl,p)),\Ima(W(\gl,p))\ra^\perp$ are invariant under $\mon_{P}(\gl\of)$. That sphere intersects the limiting curvature sphere of the $\gl$-Calapso transform $(\la f_{\gl,p}\ra,\pb\Qua)$ normalized at $p$ orthogonally in the limit point $\la f_{\gl,p}(\sip)\ra$ (cf. Thm \ref{theorem:cal_trafo_surface_so}).

Moreover, $(\mon_{P}(\gl\of))^j=\id$ if and only if $j\sqrt{1-2\gl}\in\mathbb N$.
\end{corollary}
\begin{proof}
The expression \eqref{eq:monodromy_secondorder_mink_firstkind} with \eqref{eq:defi_wpm_s_mink} is obtained directly from \eqref{eq:monodromy_first_kind}. If $(\mon_P(\gl\of))^j\neq \id$, then it is a Euclidean rotation with fixed point set
$$\pb\hyfr(p)\gG_{p}^{~\sip}(\Gtr{\pb\ppf_\gl}{p}{\gatr {\pb\hyfr} {\gl\pb\of}})\la \Rea(W(\gl),\Ima(W(\gl))\ra^\perp.$$
The null lines in this set are the points of an $(n-2)$-dimensional sphere that intersects the limiting curvature sphere of $\la f_{\gl,p}\ra$ at $\sip$ orthogonally, as can be seen from \eqref{eq:lim_curvature_sphere_caltrafo_so}. Both these spheres contain the limit point 
$$\la f_{\gl,p}(\sip)\ra=\lim_{\comp\ni q\rightarrow \sip}\pb\hyfr(p)\gG_{ p}^{~q}(\Gtr{\pb\ppf_\gl}{p}{\gatr {\pb\hyfr} {\gl\pb\of}})\gG_{p}^{~q}(\pb\ppf_\gl)\la \baszer\ra= \pb\hyfr(p)\gG_{ p}^{~q}(\Gtr{\pb\ppf_\gl}{ p}{\gatr {\pb\hyfr} {\gl\pb\of}})\la V_+\ra.$$
and hence they intersect orthogonally in that point.
\end{proof}

\begin{corollary}\label{cor:monodromy_secondorder_firstkind_mink}
Let $(\la f\ra,\Qua)$ have a pole of second order at $s$. For $\gl\in\R^\times$ with $0<1-2\gl<1$, let $(\la f_{\gl,p}\ra,\pb\Qua)$ be the $\gl$-Calapso transform of $(\la \pb f\ra,\pb\Qua)$ normalized at $p\in\pb\gD\bsb$. 
%For $n\in\mathbb N$, denote by $(\gD\bsb)_n$ the $n$-fold cover of $\gD\bsb$ and by $\proj_n:\pb\gD\bsb\rightarrow (\gD\bsb)_n$ the corresponding projection.
The pushforward of $\la f_{\gl,p}\ra$ to the $j$-fold cover of $\gD\bsb$ exists if and only if $j\sqrt{1-2\gl}$ is an integer.
\end{corollary}
\begin{proof}
By Prop \ref{prop:monodromy_dar_cal_trafo}, the pushforward of $\la f_{\gl,p}\ra$ to the $j$-fold cover of $\gD\bsb$ exists if and only if every point in the image of $\la f_{\gl,p}\ra$ is invariant under $(\mon_{\proj(p)}(\gl\of))^j$. When $j\sqrt{1-2\gl}$ is an integer, $(\mon_{\proj(p)}(\gl\of))^j=\id$ by Cor \ref{cor:monodromy_sop} and all points in $S^n$ are invariant under $\id$. But when $j\sqrt{1-2\gl}$ is not an integer, then $(\mon_{\proj(p)}(\gl\of))^j\neq \id$ and the set of points invariant under $(\mon_{\proj(p)}(\gl\of))^j$ form an $(n-2)$-dimensional sphere which is orthogonal to the limiting curvature sphere of $\la f_{\gl,p}\ra$ at $\sip$, by Cor \ref{cor:monodromy_sop}. If the image of $\la f_{\gl,p}\ra$ was contained in that $(n-2)$-dimensional sphere, it would intersect its limiting curvature sphere orthogonally, which is certainly not possible.
\end{proof}
\begin{corollary}
Let $(\la f\ra,\Qua)$ have a pole of second order at $s$ and $(\la \fh\ra,\pb\Qua)$ be a $\gl$-Darboux transform of $(\la \pb f\ra,\pb \Qua)$ with $0<1-2\gl<1$. The pushforward of $\la\fh\ra$ to the $j$-fold cover of $\gD\bsb$ exists if and only if for any $p\in\pb\gD\bsb$ the point $\la \fh(p)\ra$ lies on the sphere
$$\la \Rea(W(\gl,p)),\Ima(W(\gl,p))\ra^\perp\cap\Li^{n+1}$$
or $j\sqrt{1-2\gl}$ is an integer, where $W(\gl,p)$ is as in \eqref{eq:defi_wpm_s_mink}.
\end{corollary}
\begin{proof}
Again, this follows from Prop \ref{prop:monodromy_dar_cal_trafo} and Cor \ref{cor:monodromy_sop}.
\end{proof}

When $1<1-2\gl$, the pure pole form $\ppf_\gl$ is of the second kind. In that case, from Prop \ref{prop:strucutre_monodromy_pfsk} and Thm \ref{theorem:cal_trafo_surface_so} we can conclude that $\la f_{\gl,p}(\sip)\ra$ is an eigendirection of $(\mon_{\proj(p)}(\gl\of))^j$ for all $j\in\mathbb N$. Therefore, the pushforward of the Darboux transform with $\la \fh(p)\ra=\la f_{\gl,p}(\sip)\ra$ to $\gD\bsb$ is well defined. However, when $j\sqrt{1-2\gl}$ is not an integer, there may be an entire $(n-2)$-dimensional sphere in $S^n$ invariant under $\mon_{\proj(p)}(\gl\of)$ and if $j\sqrt{1-2\gl}$ is an integer, it may be that $\mon_{\proj(p)}(\gl\of)$ is the identity. It depends on that, whether there are further Darboux transforms which can be pushed forward to the $j$-fold cover of $\gD\bsb$ and, similarly, whether the pushforward of $\la f_{\gl,p}\ra$ is well defined.

%--------------------------------------------------------------------------------------------
%---------------------- POSITIVE LAMBDA  ----------------------------------------------------
%--------------------------------------------------------------------------------------------

\subsection{The behaviour at and around the singularity for $1-2\gl<0$}
\subsubsection{Limiting behaviour of transforms}

For spectral parameter $\gl$ such that $1-2\gl<0$, the pure pole form $\ppf_\gl$ given by \eqref{eq:form_of_xila_so_s} is spacelike. In particular, it is of the first kind and we can apply the results of Prop \ref{prop:maintool_firstkind_surf}.

\begin{theorem}\label{theorem:cal_trafo_so_posla_s}
Let $(\la f\ra,\Qua)$ have a pole of second order at $s$. For $\gl\in\R^\times$ with $1-2\gl<0$, let $(\la f_{\gl,p}\ra,\pb\Qua)$ be the $\gl$-Calapso transform of $(\la \pb f\ra,\pb\Qua)$ normalized at $p\in\pb\gD\bsb$. Choose $\comp\subset\pb\gD\bsb$ such that $\comp\cub$ is compact with limit point $\sip$. Then, as one approaches $\sip$ in $\comp$, the surface $\la f_{\gl,p}\ra$ tends towards
$$\la S_{\gl,p}\ra:=\pb\hyfr(p)\gG_{p}^{~\sip}(\Gtr {\pb\ppf_\gl}{p}{\gatr{\pb\hyfr}{\gl\pb\of}})\gG_p(\pb\ppf_\gl)\la \baszer\ra,$$
which is a parametrization of the universal cover of the limit set
\begin{equation}
\LiSe[\la f_{\gl,p}\ra]= \pb\hyfr(p)\gG_p^{~\sip}(\Gtr {\pb\ppf_\gl}{p}{\gatr{\pb\hyfr}{\gl\pb\of}})\Big(\left(\la \baszer,\bastanu,\bastanv,\basinf\ra\cap\Li^{n+1}\right) \bs\{\la W_+\ra,\la W_-\ra\}\Big),
\label{eq:formoflimitset_caltrafo_so}
\end{equation}
a 2-sphere with two points removed.
\end{theorem}
\begin{proof}
With the orthogonal lifts $\pb G$ and $\pb\Ppf_\gl$ of $\pb g$ and $\pb\ppf_\gl$, respectively, consider the lifts
\begin{align}
f_{\gl,p}&=\pb G(p)\gG_p(\Gtr {\pb\Ppf_\gl}{p}{\gatr{\pb G}{\gl\pb\Of}})\gG_p(\pb\Ppf_\gl)\,\baszer,\notag\\
S_{\gl,p}&=\pb G(p)\gG_p^{~\sip}(\Gtr {\pb\Ppf_\gl}{p}{\gatr{\pb G}{\gl\pb\Of}})\gG_p(\pb\Ppf_\gl)\,\baszer\label{eq:limitingmap_caltrafo_secord}
\end{align}
of $\la f_{\gl,p}\ra$ and $\la S_{\gl,p}\ra$, respectively. From the explicit form \eqref{eq:expo_xila_so_s} of $\gG_p(\pb\Ppf_\gl)$, we can deduce that $\gG_p(\pb\Ppf_\gl)\,\baszer$ and therefore also the lifts \eqref{eq:limitingmap_caltrafo_secord} are bounded from below with respect to $|\cdot|$ on $\comp$. Using \eqref{eq:convergence_primofomegati} and the boundedness \eqref{eq:upbound_ex_spacelike_s} of $\gG_p(\pb\Ppf_\gl)$ on $\comp$, one easily finds that the inner product $\ipl f_{\gl,p},S_{\gl,p}\ipr$ tends to zero as $\sip$ is approached in $\comp$ and thus $\la f_{\gl,p}\ra$ approaches $\la S_{\gl,p}\ra$. Hence, the limit set of $\la f_{\gl,p}\ra$ agrees with that of $\la S_{\gl,p}\ra$. But since $\la S_{\gl,p}\ra$ is periodic, its limit set is equal to its image. To see that the image of $\la S_{\gl,p}\ra$ is given by \eqref{eq:formoflimitset_caltrafo_so}, we use again \eqref{eq:expo_xila_so_s} to find that
\begin{align*}
\ima\left(\gG_p(\pb\ppf_\gl)\la\baszer\ra\right)=&\left\{\la \ga V_++\bar\ga V_-+\beta W_++\frac{1}{\beta}W_-\ra~|~\ga\in\Com,~\beta\in\R^\times,~|\ga|^2=-\frac{\ipl W_+,W_-\ipr}{\ipl V_+,V_-\ipr}\right\}\\
=&\left(\la \baszer,\bastanu,\bastanv,\basinf\ra\cap\Li^{n+1}\right) \bs\{\la W_+\ra,\la W_-\ra\}.
\end{align*}
\end{proof}

A generic Darboux transform again converges to $\la f(s)\ra$, as proved in the following
\begin{theorem}
Let $(\la f\ra,\Qua)$ have a pole of second order at $s$. For $\gl\in\R^\times$ with $1-2\gl<0$, let $\LiSe[\la f_{\gl,p}\ra]$ be the limit set of the $\gl$-Calapso transform of $(\la \pb f\ra,\pb\Qua)$ normalized at $p\in\pb\gD\bsb$ and $(\la\fh\ra,\pb\Qua)$ be a $\gl$-Darboux transform of $(\la \pb f\ra,\pb\Qua)$ such that $\la \fh(p)\ra$ does not lie in $\LiSe[\la f_{\gl,p}\ra]$. Choose $\comp\subset\pb\gD\bsb$ such that $\comp\cub$ is compact with limit point $\sip$. Then, the restriction of $\la \fh\ra$ to $\comp$ has the limit $\la f(\sip)\ra$ at $\sip$.
\end{theorem}
\begin{proof}
Apart from a small subtlety, the proof is analogous to that of \cite[Prop 3]{fu18p}. Let $z$ be a coordinate adapted to $\sqrt{\Qua}$ and think of $p\in\pb\gD\bsb$ fixed. With $\la\Fh_p\ra:=\pb\hyfr(p)^{-1}\la\fh(p)\ra$, we use the factorization \eqref{eq:factorization_into_surf} to write
$$\la \fh\ra=\pb\hyfr \,\la u\ra~~~\text{where}~~~\la u\ra:= \gG^{~p}(\gatr {\pb\hyfr} {\gl\pb\of})\la \Fh_p\ra=\gG^{~p}(\pb\ppf_\gl)\gG^{~p}(\Gtr{\pb\ppf_\gl}p{\gatr{\pb\hyfr}{\gl\pb\of}})\la \Fh_p\ra.$$
The singular frame $ \hyfr$ has limit $\la f(\sip)\basinf^*\ra$ at $\sip$ by \eqref{eq:limit_singtrafo} while $\gG^{~p}(\gatr {\pb\hyfr} {\gl\pb\of})\la \Fh_p\ra$ does not have a limit at $\sip$. To show that nevertheless the restriction of $\la \fh\ra$ to $\comp$ has a limit at $\sip$, we want to apply \cite[Lemma 4]{fu18p}. Thus we convince ourselves that there is an $\ti r\in\R$ and a neighbourhood $U\subset S^n$ of $\la \basinf\ra$ such that
\begin{equation}
\forall q\in\comp~\text{with}~|\pb z(q)|\leq \ti r:~~~\la u(q)\ra\notin U.
\label{eq:condition_for_lemma_poslasurf}
\end{equation}
To prove this, we note that as one approaches $\sip$ in $\comp$, the map $\la u\ra$ approaches $\la \ti u\ra$ defined by
$$\la \ti u\ra:=\gG^{~p}(\pb\ppf_\gl)\gG_s^{~p}(\Gtr{\pb\ppf_\gl}p{\gatr{\pb\hyfr}{\gl\pb\of}})\la \Fh_p\ra$$
because, by \eqref{eq:convergence_primofomegati}, $\lqbc\ipl u(q),\ti u(q)\ipr=0$ for the lifts
$$u=\gG^{~p}(\pb\Ppf_\gl)\gG^{~p}(\Gtr{\pb\Ppf_\gl}p{\gatr{\pb G}{\gl\pb\Of}}) \Fh_p  ,~~~\ti u=\gG^{~p}(\pb\Ppf_\gl)\gG_s^{~p}(\Gtr{\pb\Ppf_\gl}p{\gatr{\pb G}{\gl\pb\Of}}) \Fh_p,$$
which are bounded from below with respect to $|\cdot|$. Under our assumption on $\la \fh(p)\ra$, this limiting map $\la \ti u\ra$ does not contain $\la \basinf\ra$ in its image. Namely, we have the equivalence
\begin{equation}
\la \ti u(q)\ra=\la\basinf\ra~~~\Leftrightarrow~~~\la \fh(p)\ra=\pb\hyfr(p) \gG_p^{~\sip}(\Gtr {\pb\ppf_\gl}{p}{\gatr{\pb\hyfr}{\gl\pb\of}})\gG_p^{~q}(\pb\ppf_\gl)\la \basinf\ra
\label{eq:away_from_nei}
\end{equation}
for all $q\in\pb\gD\bsb$. From \eqref{eq:formoflimitset_caltrafo_so} we see that $\pb\hyfr(p) \gG_p^{~\sip}(\Gtr {\pb\ppf_\gl}{p}{\gatr{\pb\hyfr}{\gl\pb\of}})\gG_p^{~q}(\pb\ppf_\gl)\la \basinf\ra$ is a point on $\LiSe[\la f_{\gl,p}\ra]$ while, by assumption, $\la \fh(p)\ra \notin \LiSe[\la f_{\gl,p}\ra]$. Thus, we can conclude from \eqref{eq:away_from_nei} that indeed $\la\basinf\ra$ does not lie in the image of $\la \ti u\ra$. Finally, since $\la u\ra$ approaches $\la \ti u\ra$ as one approaches $\sip$ in $\comp$, the limit set of $\la u\ra$ does not contain $\la \basinf\ra$ and so there is an $\ti r\in(0,r_0]$ such that \eqref{eq:condition_for_lemma_poslasurf} holds. We can thus apply \cite[Lemma 4]{fu18p} to complete the proof.
\end{proof}

The following theorem treats the case of those $\gl$-Darboux transforms for which $\la \fh(p)\ra$ lies on the limit set of the $\gl$-Calapso transform normalized at $p$. Since it can be proved in complete analogy to the first two parts of \cite[Thm 6]{fu18p}, we only sketch its proof very briefly.
\begin{theorem}
Let $(\la f\ra,\Qua)$ have a pole of second order at $s$ and $(\la \fh\ra,\pb\Qua)$ be a $\gl$-Darboux transform of $(\la \pb f\ra,\pb\Qua)$ with $1-2\gl<0$ and $\la \fh(p)\ra\in \LiSe[\la f_{\gl,p}\ra]$. For any $\comp\subset\pb\gD\bsb$ such that $\comp\cub$ is compact with limit point $\sip$, as one approaches $s$ in $\comp$, the surface $\la \fh\ra$ approaches the curvature sphere of $\la f\ra$ at $\sip$, but does not have a limit.
\end{theorem}
\begin{proof}
Using the frame and lifts of normal fields $N_i$ of Lemma \ref{lemma:nice_frame_s}, one can show that $\ipl \fh,N_i\circ\proj\ipr$ converges to zero as one approaches $\sip$ inside $\comp$ for a lift $\fh$ that is bounded from below with respect to $|\cdot|$. Thus, $\la \fh\ra$ approaches the curvature sphere of $\la f\ra$ at the umbilic $\sip$. To show that both $\la f(\sip)\ra=\la F(s)\baszer\ra$ and $\la F(s)\basinf\ra$ are limit points of $\la \fh\ra$, one may use the form \eqref{eq:formoflimitset_caltrafo_so} of the limit set $\LiSe[\la f_{\gl,p}\ra]\ni \la \fh(p)\ra$ to conclude that there are sequences $(q_i)_{i\in\mathbb N}$ and $(\ti q_i)_{i\in\mathbb N}$ in $\comp$ with limit $\sip$ such that
$$\forall i\in\mathbb N:~~~\gG_p^{~q_i}(\pb\ppf_\gl)\gG_s^{~p}(\Gtr {\pb\ppf_\gl}{p}{\gatr{\pb\hyfr}{\gl\pb\of}})\pb\hyfr(p)^{-1}\la \fh(p)\ra=\la\baszer\ra,~~~\gG_p^{~\ti q_i}(\pb\ppf_\gl)\gG_s^{~p}(\Gtr {\pb\ppf_\gl}{p}{\gatr{\pb\hyfr}{\gl\pb\of}})\pb\hyfr(p)^{-1}\la \fh(p)\ra=\la\basinf\ra.$$
As in the second part of the proof of \cite[Thm 6]{fu18p} one can then show that $\la \fh(q_i)\ra$ converges to $\la f(\sip)\ra$ and $\la\fh(\ti q_i)\ra$ converges to $\la F(s)\basinf\ra$ as $i$ tends to infinity.
\end{proof}
%------------------------------------------------------------------------------------------------
%----------------------------------------- M O N O D R O M Y ------------------------------------
%------------------------------------------------------------------------------------------------
\subsubsection{Monodromy of transforms}\label{section:surf_monodromy_so_negla}
For $1-2\gl<0$, the pole form $\gatr \hyfr {\gl\of}$ is of the first kind. An expression for its monodromy is thus obtained from Prop \ref{prop:maintool_firstkind_surf}.

\begin{corollary}\label{cor:monodromy_sop}
Let $(\la f\ra,\Qua)$ have a pole of second order at $s$. Denote by $\of$ the 1-form associated to $(\la f\ra,\Qua)$. For $\gl\in\R$ with $1-2\gl<0$ and points $P\in\gD\bsb$ and $p\in\pb\gD\bsb$ with $\proj(p)=P$, the monodromy of $\gl\of$ with base point $P$ is
\begin{equation}
\mon_{P}(\gl\of)=\La \Pi_{\la W_+(\gl,p),W_-(\gl,p)\ra^\perp}+e^{-2\pi \sqrt{2\gl-1}}\Pi_{\la W_+(\gl,p)\ra}+e^{2\pi \sqrt{2\gl-1}}\Pi_{\la W_-(\gl,p)\ra}\Ra,
\end{equation}
where 
\begin{equation}
\la W_\pm(\gl,p)\ra=\pb\hyfr(p)\gG_{p}^{~\sip}(\Gtr{\pb\ppf_\gl}{p}{\gatr {\pb\hyfr} {\gl\pb\of}})\la W_\pm(\gl)\ra,~~~~W_\pm(\gl)=\sqrt{2\gl-1}\,\bastanv\pm(\gl\baszer+\basinf-\bastanu),
\label{eq:defi_wpm_s}
\end{equation}
are two points on the limit sphere of the $\gl$-Calapso transform normalized at $p$.

In particular, $\la W_\pm(\gl,p)\ra$ are the only null eigenspaces of $\mon_P(\gl\of)^j$ for all $j\in\mathbb N$.
\end{corollary}
We remark that, by \eqref{eq:formoflimitset_caltrafo_so}, the points $\la W_\pm(\gl,p)\ra$ lie on the limit sphere, but not on the limit set of the $\gl$-Calapso transform normalized at $p$.
\begin{corollary}\label{cor:monodromy_secondorder}
Let $(\la f\ra,\Qua)$ have a pole of second order at $s$. No $\gl$-Calapso transform of $(\la \pb f\ra,\pb\Qua)$ with $1-2\gl<0$ can be pushed forward to the $j$-fold cover of $\gD\bsb$ for any $j\in\mathbb N$.
\end{corollary}
\begin{proof}
This follows readily from Prop \ref{prop:monodromy_dar_cal_trafo}: no $\gl$-Calapso transform is constant. Therefore, the image of any $\gl$-Calapso transform contains infinitely many points. But for any $j\in\mathbb N$ there are only two points, $\la W_\pm(\gl,p)\ra$, which are invariant under $(\mon_{\proj(p)}(\gl\of))^j$, by Cor \ref{cor:monodromy_sop}.

\end{proof}
\begin{corollary}
Let $(\la f\ra,\Qua)$ have a pole of second order at $s$. The pushforward of a $\gl$-Darboux transform $(\la \fh\ra,\pb\Qua)$ of $(\la\pb f\ra,\pb\Qua)$ with $1-2\gl<0$ to the $j$-fold cover of $\gD\bsb$ is well defined if and only if for any $p\in\pb\gD\bsb$
$$\la \fh(p)\ra\in \{\la W_+(\gl,p)\ra,\la W_-(\gl,p)\ra\}.$$

\end{corollary}
\begin{proof}
Again, this follows from Prop \ref{prop:monodromy_dar_cal_trafo} and Cor  \ref{cor:monodromy_sop}.
\end{proof}

\begin{figure}[h]
	\centering
	~~~~~~\begin{subfigure}[t]{0.45\textwidth}
		\centering
		\includegraphics[height=100pt]{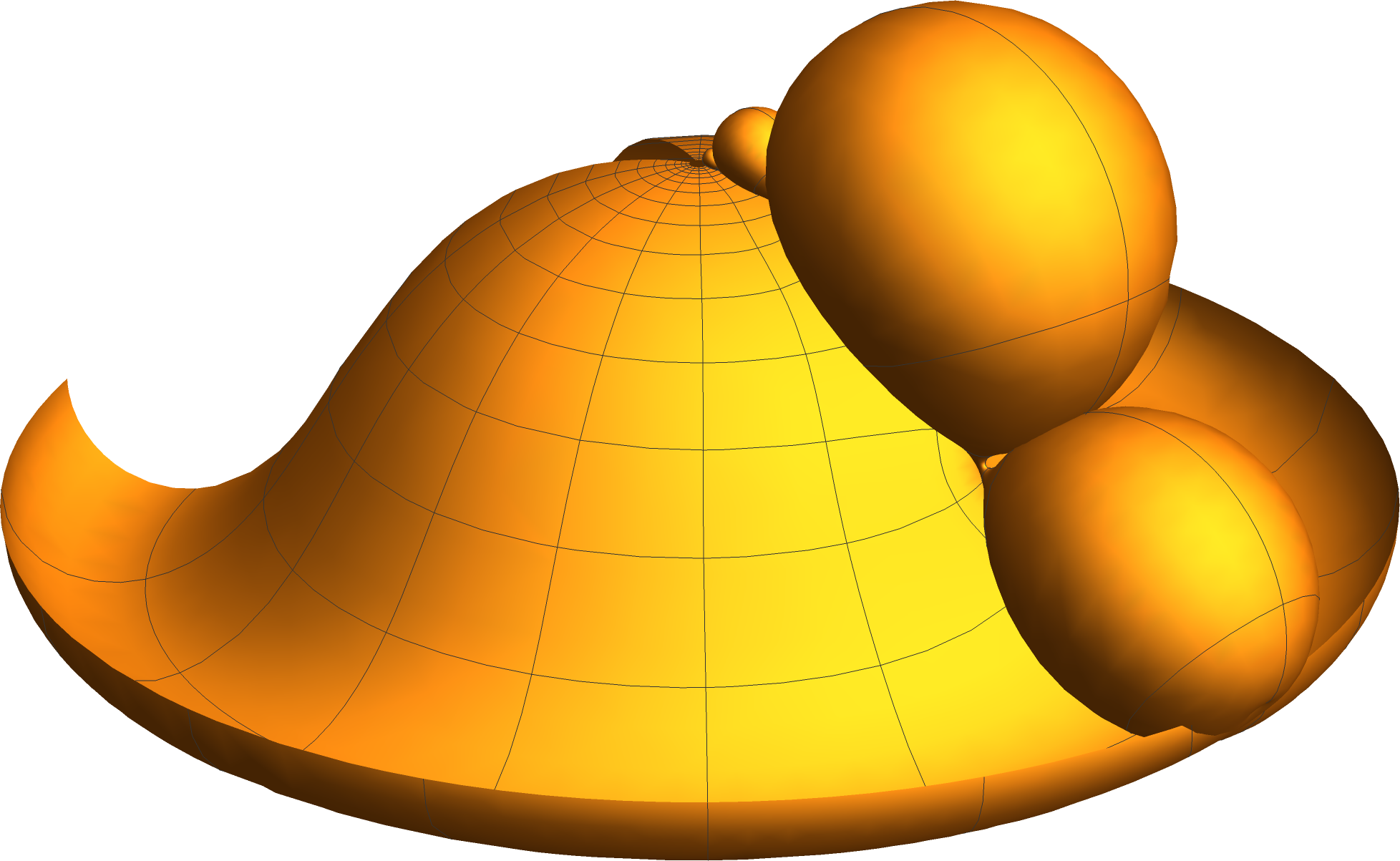}
	\end{subfigure}~~~~~~~~
	\begin{subfigure}[t]{0.45\textwidth}
		\centering
		\includegraphics[height=120pt]{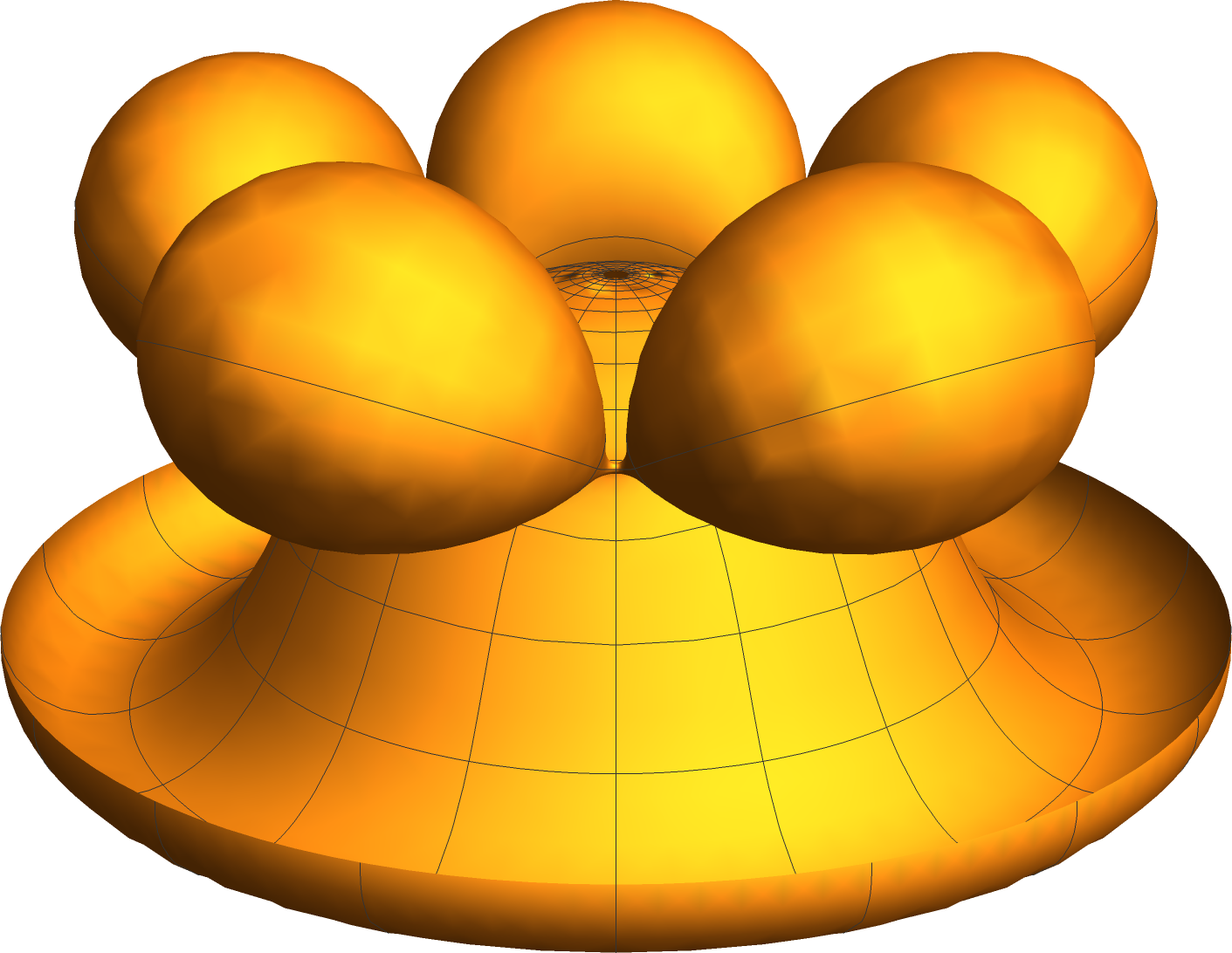}
	\end{subfigure}
	\caption{Darboux transforms of a surface of revolution with spectral parameter $\gl$ such that $1-2\gl<0$ and $1-2\gl>0$ on the left and  right hand side, respectively.}\label{fig:darsurfrev}
\end{figure}

\bibliographystyle{alpha}

\end{document}